\documentclass[reqno, letterpaper, oneside]{article}
\usepackage{amsmath,amsfonts,amssymb,amsthm}
\usepackage[final]{graphicx}
\usepackage[usenames,dvipsnames]{color}

\usepackage{bbm}
\usepackage{setspace}
\usepackage{geometry}
\usepackage{enumerate}

\usepackage{algorithm, algorithmic}
\makeatletter
\renewcommand{\fnum@algorithm}{} 
\makeatother

\usepackage{hyperref}

\newtheorem{theorem}{Theorem}
\newtheorem{lemma}[theorem]{Lemma}
\newtheorem{corollary}[theorem]{Corollary}

\newtheorem{remark}[theorem]{Remark}
\newtheorem{problem}{Problem}
\theoremstyle{definition}

\newcommand{\refT}[1]{Theorem~\ref{#1}}
\newcommand{\refC}[1]{Corollary~\ref{#1}}

\newcommand{\refL}[1]{Lemma~\ref{#1}}
\newcommand{\refR}[1]{Remark~\ref{#1}}
\newcommand{\refS}[1]{Section~\ref{#1}}

\newcommand{\refApp}[1]{Appendix~\ref{#1}}
\newcommand\RR{{\mathbb R}}

\newcommand\E{\operatorname{\mathbb E{}}}
\renewcommand\Pr{\operatorname{\mathbb P{}}}
\renewcommand\P{\operatorname{\mathbb P{}}}

\newcommand\Var{\operatorname{Var}}

\newcommand\Bin{\operatorname{Bin}}

\newcommand\set[1]{\ensuremath{\{#1\}}}
\newcommand\bigset[1]{\ensuremath{\bigl\{#1\bigr\}}}

\newcommand\lrset[1]{\ensuremath{\left\{#1\right\}}}

\newcommand\bigpar[1]{\bigl(#1\bigr)}
\newcommand\Bigpar[1]{\Bigl(#1\Bigr)}
\newcommand\biggpar[1]{\biggl(#1\biggr)}
\newcommand\lrpar[1]{\left(#1\right)}
\newcommand\bigsqpar[1]{\bigl[#1\bigr]}
\newcommand\Bigsqpar[1]{\Bigl[#1\Bigr]}
\newcommand\biggsqpar[1]{\biggl[#1\biggr]}

\newcommand\bigcpar[1]{\bigl\{#1\bigr\}}
\newcommand\Bigcpar[1]{\Bigl\{#1\Bigr\}}

\newcommand\abs[1]{|#1|}
\newcommand\bigabs[1]{\bigl|#1\bigr|}
\newcommand\Bigabs[1]{\Bigl|#1\Bigr|}

\def\rompar(#1){\textup(#1\textup)}    

\def\xexp(#1){e^{#1}}
\newcommand\ceil[1]{\lceil#1\rceil}
\newcommand\bigceil[1]{\bigl\lceil#1\bigr\rceil}

\newcommand\floor[1]{\lfloor#1\rfloor}

\newcommand\noproof{\qed}

\newcommand{\eps}{\epsilon}

\newcommand{\hc}{\hat{c}}

\def\he{\hat{e}}
\def\hq{\hat{q}}
\def\hy{\hat{y}}
 \def\D{\Delta}

\newcommand{\cE}{\mathcal{E}}
\newcommand{\cF}{\mathcal{F}}
\newcommand{\cG}{\mathcal{G}}
\newcommand{\cH}{\mathcal{H}}
\newcommand{\cI}{\mathcal{I}}

\newcommand{\cK}{\mathcal{K}}

\newcommand{\cN}{\mathcal{N}}

\newcommand{\cP}{\mathcal{P}}

\newcommand{\cS}{\mathcal{S}}

\newcommand{\cU}{\mathcal{U}}

\newcommand{\cB}{\mathcal{B}}
\newcommand{\cC}{\mathcal{C}}
\newcommand{\cX}{\mathcal{R}}

\newcommand{\cc}{\tau} 
\newcommand{\balp}{b} 
\newcommand{\beps}{\sigma} 

\newcommand{\ca}{\sigma} 

\newcommand{\Gnp}{G_{n,p}}
\newcommand{\Gnpp}[1]{G_{n,{#1}}}

\newcommand{\deps}{{\eps^2}}
\newcommand{\kr}{r} 

\newcommand{\ci}{\chi'} 
\newcommand{\ccn}{\text{cc}} 
\newcommand{\cct}{\text{cc}_{\Delta}} 
\newcommand{\ccc}{\text{cc}'} 
\newcommand{\cpn}{\text{cp}} 

\newcommand{\fpra}{\operatorname{dim}_{\mathrm{P}}}

\renewcommand{\Gamma}{\mathcal{K}} 
\renewcommand{\emptyset}{\varnothing} 

\newcommand{\indic}[1]{\mathbbm{1}_{\{{#1}\}}}

\let\OLDthebibliography\thebibliography
\renewcommand\thebibliography[1]{
  \OLDthebibliography{#1}
  \setlength{\parskip}{0pt}
  \setlength{\itemsep}{0pt plus 0.3ex}
}

\title{Prague dimension of random graphs} 
\author{He Guo and Kalen Patton and Lutz Warnke%
\thanks{School of Mathematics, Georgia Institute of Technology, Atlanta GA~30332, USA. 
E-mail: {\tt he.guo@gatech.edu, kpatton33@gatech.edu, warnke@math.gatech.edu}.
Research supported by NSF grant DMS-1703516, NSF~CAREER grant~DMS-1945481, and a Sloan Research Fellowship.}}
\date{November 17, 2020}

\begin{document}
\maketitle

\begin{abstract}
The Prague~dimension of graphs was introduced by Ne\v{s}et\v{r}il, Pultr and R\"{o}dl in the 1970s. 
Proving a conjecture of \mbox{F\"{u}redi} and \mbox{Kantor}, we show that the Prague~dimension of the binomial random~graph is typically of order~$n/\log n$ for constant edge-probabilities. 
The main new proof ingredient is a Pippenger--Spencer~type edge-coloring result for random~hypergraphs with large~uniformities, i.e., edges of size~$O(\log n)$. 
\end{abstract}

\section{Introduction}
Various notions of dimension are important in many areas of mathematics, as a measure for the complexity of objects. 
For graphs, one interesting notion of dimension was introduced by Ne\v{s}et\v{r}il, Pultr and R\"{o}dl~\cite{NR1976,NP1977} in the~1970s.  
The Prague dimension $\fpra(G)$ of a 
graph~$G$ (also called product dimension) 
is the minimum number~$d$ such that~$G$ is an induced subgraph of the product of~$d$ complete graphs. 
There are many equivalent definitions of~$\fpra(G)$, see~\cite{W1996,HN2004,AA2020}, 
indicating that this is a natural combinatorial notion of dimension~\cite{LNP1980,HN2004,SB2019}, 
which in fact has appealing connections with efficient representations of graphs~\cite{W1996,K2010,FK2018}.

Despite receiving considerable attention during the last 40~years 
 (including combinatorial~\cite{ER1996,F2000}, information theoretic~{\cite{KO1998,KM2001}} and algebraic~{\cite{NR1976,LNP1980,Alon1986,AA2020} approaches), 
the Prague dimension is still not well understood, i.e., its determination usually remains
a notoriously\footnote{The decision problem of whether~$\fpra(G) \le k$ holds is also known to be NP-complete for~$k \ge 3$, see~\cite{NP1977}.} 
difficult task~\cite{ER1996,FK2018}. 
To gain further insight into the behavior of this intriguing graph parameter, 
it thus is natural and instructive to investigate the Prague dimension of random graphs, 
as initiated by Ne\v{s}et\v{r}il and R\"{o}dl~\cite{NR1983} already in the~1980s. 
For the binomial random graph~$\Gnp$, 
F\"{u}redi and Kantor conjectured that 
with high probability\footnote{As usual, we say that an event holds~\emph{whp} (with~high~probability) if it holds with probability tending to~$1$ as~$n\to \infty$.} (whp) 
the order is $\fpra(\Gnp) =\Theta(n/\log n)$ 
for constant edge-probabilities~$p$, see~\cite[Conjecture~15]{FK2018} and~\cite{OWF2017}.

In this paper we prove the aforementioned F\"{u}redi--Kantor Prague dimension conjecture, 
by showing that the binomial random graph whp satisfies~$\fpra(\Gnp) =\Theta(n/\log n)$ for constant edge-probabilities~$p$. 
\begin{theorem}[Prague dimension of random graphs]\label{thm:prague} 
For any fixed edge-probability~$p \in (0,1)$ there are constants~$c,C>0 $ so that the Prague dimension of the random graph~$\Gnp$ satisfies with high probability 
\begin{equation}\label{eq:prague:bounds}
c \frac{n}{\log n} \le \fpra(\Gnp) \le C\frac{n}{\log n} .
\end{equation}
\end{theorem}
\noindent
The Prague dimension of $n$-vertex graphs can be as large as~$n-1$, see~\cite{LNP1980,W1996}, 
so an important consequence of \refT{thm:prague} is that 
almost all $n$-vertex graphs have a significantly smaller Prague dimension of order~$n/\log n$  
(this follows since the random graph~$\Gnpp{1/2}$ is uniformly distributed over all $n$-vertex graphs).

For our purposes it will be useful to view the Prague dimension as a clique~covering and coloring~problem, 
and this convenient perspective hinges on the following equivalent definition~\cite{W1996,AA2020}: 
that $\fpra(G)$ equals the minimum number of subgraphs of the complement~$\overline{G}$ of~$G$ 
such that (i)~each subgraph is a vertex-disjoint union of cliques, and
(ii)~each edge of~$\overline{G}$ is contained in at least one of the subgraphs, but not all of~them.

Our main contribution is the upper bound on the Prague dimension in~\eqref{eq:prague:bounds}, whose proof 
carefully combines two different random greedy approaches: 
firstly, a semi-random `nibble-type' algorithm to \mbox{iteratively} \mbox{decompose} the 
edges of~$\overline{\Gnp}$ into edge-disjoint cliques of size~$O(\log n)$, 
and, secondly, a random greedy coloring algorithm to regroup these 
cliques into~$O(n/\log n)$ subgraphs of~$\overline{\Gnp}$ 
consisting of vertex-disjoint cliques, 
which together eventually gives~$\fpra(\Gnp)=O(n/\log n)$; 
see \refS{sec:strategy} for more details. 
Interestingly, this combination allows us to exploit 
the best features of both greedy approaches: 
the semi-random 
approach makes it easier to guarantee 
certain pseudo-random properties in the first decomposition step,  
and the random greedy approach makes it easier to guarantee 
that all cliques are properly colored in the second regrouping step 
(which in fact requires the pseudo-random properties established in the first~step).

One major obstacle for this natural proof approach is that the cliques have size~$O(\log n)$, 
which makes many standard tools and techniques unavailable, as they are usually restricted to objects of constant size. 
Notably, in order to overcome this technical~difficulty in the second regrouping step, in this paper we develop a new 
Pippenger--Spencer type coloring result for random hypergraphs with edges of size $O(\log n)$, 
which we believe to be of independent interest; see~\refS{sec:CI}. 
Beyond Prague dimension and hypergraph coloring results, 
further contributions of this paper include the proof of a related conjecture of F\"{u}redi and Kantor~\cite{FK2018}, 
and a strengthening of an old edge-covering result of Frieze and Reed~\cite{FR1995}; see~\refS{sec:cp}.

\subsection{Chromatic index of random subhypergraphs}\label{sec:CI}
Coloring problems play an important role in much of combinatorics, 
and in our Prague dimension proof one key ingredient 
also corresponds to a hypergraph coloring result. 
The chromatic index~$\ci(\cH)$ of a hypergraph~$\cH$ 
is the smallest number of colors needed to properly color its edges, 
i.e., so that no two intersecting edges receive the same color. 
Writing~$\Delta(\cH)$ for the maximum degree, we trivially have~$\ci(\cH) \ge \Delta(\cH)$. 
Vizing's theorem from the 1960s states that ${\ci(G) \le {\Delta(G)+1}}$ for any graph~$G$. 
Influential work of Pippenger and Spencer~\cite{PS1989} from the~1980s gives 
a partial extension to $\kr$-uniform hypergraphs~$\cH$: 
for any~$\delta>0$ they showed that ${\ci(\cH) \le {(1+\delta) \Delta(\cH)}}$ 
for any nearly regular~$\cH$ with 
small~codegrees and edges of size~${\kr=\Theta(1)}$.

It is challenging to extend the Pippenger--Spencer coloring arguments 
to edges of size~${\kr=O(\log n)}$, which is what we desire in 
our main Prague dimension proof (where cliques correspond to edges of an auxiliary hypergraph). 
\refT{thm:ps:random} overcomes this size obstacle in the random setting, i.e., 
for coloring random edges of any nearly regular hypergraph~$\cH$ with 
small codegrees. 
This probabilistic Pippenger--Spencer type result indeed suffices for our purposes, 
as we shall see in Sections~\ref{sec:strategy} and~\ref{sec:packing:CI}.  
Here ${\deg_\cH(v)} := {|\{e \in E(\cH): v \in e\}|}$ and ${\deg_\cH(u,v)} := {|\{e \in E(\cH): \{u,v\} \subseteq e\}|}$ 
denote the degree and codegree, as~usual. 
\begin{theorem}[Chromatic index of random subhypergraphs]\label{thm:ps:random} 
For all reals~$\delta,\sigma,\balp>0$ with~${\balp \le \delta\sigma/30}$ 
there is~${n_0=n_0(\delta,\sigma,\balp)>0}$  
such that, for all integers~${n \ge n_0}$, ${2 \le \kr \le \balp \log n}$, ${n^{1+\sigma} \le m \le n^{\kr n^{\sigma/5}}}$ 
and all reals~${D > 0}$, the following holds 
for every $n$-vertex $\kr$-uniform hypergraph~$\cH$ satisfying 
\begin{gather}
\label{as:degcodeg}
\max_{v\in V(\cH)}|\deg_\cH(v)-D|\le n^{-\sigma}D 
\qquad \text{ and } \qquad 
\max_{u\neq v\in V(\cH)}\hspace{-0.075em}\deg_\cH(u,v)\le n^{-\sigma}D.
\end{gather} 
We have $\Pr(\ci(\cH_{m}) \le {(1+\delta)\kr m/n}) \ge 1-m^{-\omega(\kr)}$, 
where $\cH_{m}$ denotes the random subhypergraph of~$\cH$ with edges~$e_1, \ldots, e_{m}$, 
where each edge~$e_i$ is independently chosen uniformly at random from~$\cH$. 
\end{theorem}
\begin{remark}\label{rem:ps:random}
Noting~${D \cdot m/|E(\cH)|} =(1+o(1)) {\kr m/n} \gg r \log m$, 
for any real~$\eps > 0$ it is straightforward to see that the maximum degree 
satisfies~$\Delta(\cH_{m}) = (1\pm \eps) {\kr m/n}$ with probability at least~${1-m^{-\omega(\kr)}}$, say. 
\end{remark}
\noindent
As discussed, 
for this paper the key point is that \refT{thm:ps:random} permits edges of size~${r=O(\log n)}$; 
we have made no attempt to optimize the ad-hoc assumptions on the number of edges~$m$ or the $n^{-\sigma}$ approximation 
terms in~\eqref{as:degcodeg}.  
The explicit technical assumption~${\balp \le \delta\sigma/30}$  allows for some flexibility in applications: 
setting ${\balp= \delta\sigma/30}$ and ${\delta = 30\balp/\sigma}$, respectively, 
using~\refR{rem:ps:random} we readily infer that whp   
\begin{equation}\label{eq:ps:random:case}
\ci(\cH_{m}) 
\: \le \:
\begin{cases}
\bigpar{1+2\delta}  \cdot \Delta(\cH_m) \quad & \text{if $\kr=o(\log n)$,}\\
O(1) \cdot \Delta(\cH_m) & \text{if $\kr = O(\log n)$,}\\
\end{cases}
\end{equation}
which gives Pippenger--Spencer like chromatic index bounds for many non-constant edge-sizes~$\kr$; 
we believe that these are of independent interest (see also~\refC{c:ps:random}).  

We prove \refT{thm:ps:random} by showing that a simple random greedy algorithm 
whp produces the desired coloring of the random edges~$e_1, \ldots, e_m$ from~$\cH$. 
This algorithm sequentially assigns each edge~$e_i$ a random color in~$\{1, \ldots, \floor{(1+\delta)\kr m/n}\}$ 
that does not appear on some adjacent edge~$e_j$ with~$j < i$. 
This coloring algorithm is very natural: 
Kurauskas and Rybarczyk~\cite{KR2015} analyzed it when~$\cH$ is the complete \mbox{$n$-vertex} \mbox{$r$-uniform} hypergraph, 
and its idea also underpins earlier work that extends the Pippenger--Spencer result to list-colorings~\cite{K1996,MR2000}. 
Taking advantage of the random setting, our proof of \refT{thm:ps:random} uses 
differential equation method~\cite{W1995,B2009,W2019} based martingale arguments 
to show that this greedy algorithm whp properly colors the first~$m$ out of~$(1+\delta)m$ random edges. 
This `more random edges' twist enables us to sidestep some of the `last few edges' complications 
that usually arise in the deterministic setting~\cite{PS1989,K1996,MR2000},  
which is one of the reasons why our analysis can allow for edges of size~$O(\log n)$;  
see~\refS{sec:edgechr} for the~details.

\subsection{Partitioning the edges of a random graph into cliques}\label{sec:cp}
Further motivation for studying the Prague dimension comes from its close connection to the 
covering and decomposition problems that pervade combinatorics,  
one interesting \mbox{non-standard} feature being that \refT{thm:prague} 
requires usage of cliques with~$O(\log n)$ vertices, rather than just subgraphs of constant~size. 
The clique covering number~$\ccn(G)$ of a graph~$G$ (also called intersection number) 
is the minimum number of cliques in~$G$ that cover the edge-set of~$G$. 
Similarly, the clique partition number~$\cpn(G)$ 
is the minimum number of cliques in~$G$ that partition the edge-set of~$G$. 
The question of estimating these natural graph parameters was raised by Erd\H{o}s, Goodman and P\'{o}sa~\cite{EGP1966} in~1966. 
Motivated in parts by applications~\cite{Orlin1977,KSW1978,Roberts1985,CPP2016}, 
both~$\ccn(G)$ and~$\cpn(G)$ have since been extensively studied 
for many interesting graph classes, see~e.g.~\cite{Wallis1982,Alon1986,CEPW1986,EOZ1993,CV2008,CFS2014}. 

For random graphs, the study of the clique covering number was initiated in the~1980s by 
Poljak, R\"{o}dl and Turz\'{\i}k~\cite{PRT1981} and Bollob\'{a}s, Erd\H{o}s, Spencer and West~\cite{BESW1993}. 
In~1995, Frieze and Reed~\cite{FR1995} showed that 
whp~$\ccn(\Gnp) = \Theta(n^2/(\log n)^2)$ for constant edge-probabilities~$p$. 
Constructing a clique covering is certainly easier than constructing a clique partition, 
since it does not have to satisfy such a rigid constraint.
Indeed, while obviously~$\ccn(G) \le \cpn(G)$, the ratio~$\cpn(G)/\ccn(G)$
can in fact be arbitrarily large, see~\cite{EFO1988}. 
However, \refT{thm:packing} demonstrates that for most graphs 
the clique partition number and clique covering number have the same order of~magnitude. 
\begin{theorem}[Clique covering and partition number of random graphs]\label{thm:packing}
For every fixed real~$\gamma \in (0,1)$ there are constants~$c>0$ and $C=C(\gamma)> 0$ so that if the edge-probability~$p=p(n)$ satisfies ${n^{-2} \ll p \le 1-\gamma}$,  
then with high probability 
\begin{equation}\label{eq:packing:bounds}
c\frac{n^2p}{(\log_{1/p} n )^2} \le \ccn(\Gnp) \le \cpn(\Gnp) \le C\frac{n^2p}{(\log_{1/p} n)^2} .
\end{equation}
\end{theorem}
\noindent
The main contribution of~\eqref{eq:packing:bounds} is the upper bound, 
which strengthens the main result of Frieze and Reed~\cite{FR1995} 
from clique coverings to clique~partitions, and also allows for~$p=p(n) \to 0$. 
Here the mild assumption~${p \le 1-\gamma}$ turns out to be necessary, 
since \refL{lemma:lowerbounds} implies that whp~${\ccn(\Gnp)/(n^2p/(\log_{1/p} n)^2) \to \infty}$ as~$p \to 1$. 
The lower bound in~\eqref{eq:packing:bounds} is straightforward: 
it is well-known that~$\Gnp$ whp has~${m=\Theta(n^2p)}$ edges and largest clique of size~${\omega=O(\log_{1/p} n)}$, 
which gives~$\ccn(\Gnp) \ge m/\binom{\omega}{2} = \Omega(n^2p/(\log_{1/p} n)^2)$.

To gain a better combinatorial understanding of clique coverings,  
it is instructive to study additional properties besides the size, 
such as their thickness~${\cct(G) := {\min_{\cC} \Delta(\cC)}}$ and 
chromatic index~${\ccc(G) := {\min_{\cC} \ci(\cC)}}$, where the minimum is taken over all clique coverings~$\cC$ 
of the edges of~$G$ (formally thinking of~$\cC$ as a hypergraph with vertex-set~$V(G)$ and~edge-set~$\cC$). 
Notably, the parameters~$\ccc(\overline{G})$ and~$\cct(\overline{G})$ approximate 
the Prague dimension and the so-called Kneser rank of~$G$, see~\cite{FK2018}. 
In particular, we have 
\begin{equation}\label{eq:PraDimApprox}
\ccc(\overline{G}) \le \fpra(G) \le \ccc(\overline{G}) + 1, 
\end{equation}
which follows by noting that the color classes of a properly colored collection~$\cC$ of cliques 
naturally correspond to subgraphs consisting of vertex-disjoint unions of cliques 
(the~$+1$ in the upper bound is only needed to handle boundary cases where an edge is contained in cliques from all color classes); 
see~\cite{NR1976,FK2018}.

For random graphs,  F\"{u}redi and Kantor~\cite{FK2018} showed that the clique covering thickness 
is whp ${\cct(\Gnp)} ={\Theta(n/\log n)}$ for constant edge-probabilities~$p$. 
Supported by ${\cct(G) \le \ccc(G)}$ and further evidence, 
they conjectured that the clique covering chromatic index is whp also 
${\ccc(\Gnp)} = {\Theta(n/\log n)}$ for constant~$p$, see~\cite[Conjecture~17]{FK2018}. 
The following theorem proves their chromatic index conjecture in a strong form, 
allowing for $p=p(n) \to 0$. 
More importantly, \refT{thm:thickCI} and inequality~\eqref{eq:PraDimApprox} 
together imply our main Prague dimension result~\refT{thm:prague}, 
since the complement~$\overline{\Gnp}$ of~$\Gnp$ has the same distribution as~$\Gnpp{1-p}$. 
\begin{theorem}[Thickness and chromatic index of clique coverings of random graphs]\label{thm:thickCI} 
For every fixed real~${\gamma \in (0,1)}$ there are constants~$c>0$ and $C=C(\gamma)> 0$ so that if the edge-probability~$p=p(n)$ satisfies ${n^{-1}\log n \ll p \le 1-\gamma}$,  
then with high probability  
\begin{equation}\label{eq:thickCI:bounds}
c\frac{np}{\log_{1/p} n }  \le \cct(\Gnp) 
\le \ccc(\Gnp) 
\le C\frac{np}{\log_{1/p} n}  .
\end{equation}
\end{theorem}
\begin{remark}
Our proof shows that the upper bound in~\eqref{eq:thickCI:bounds} remains valid 
when the definition of~$\ccc(G)$ is restricted to clique partitions of the edges 
(instead of clique coverings); see Sections~\ref{sec:strategy} and~\ref{sec:nibble}. 
\end{remark}
\noindent
The main contribution of~\eqref{eq:thickCI:bounds} is the upper bound, 
where the mild assumption~${p \le 1-\gamma}$ again turns out to be necessary, 
since \refL{lemma:lowerbounds} implies that whp~${\cct(\Gnp)/(np/\log_{1/p} n) \to \infty}$ as~$p \to 1$. 
The lower bound in~\eqref{eq:thickCI:bounds} is straightforward: 
it is well-known that~$\Gnp$ whp has maximum degree~${\Delta=\Theta(np)}$ and largest clique of size~${\omega=O(\log_{1/p} n)}$, 
which gives~$\cct(\Gnp) \ge \Delta/(\omega -1) = \Omega(np/\log_{1/p} n)$.

\subsection{Proof strategy: finding a clique partition of a random graph}\label{sec:strategy}
We now comment on the proofs of Theorems~\ref{thm:packing}--\ref{thm:thickCI}, 
for which it remains to establish the upper bounds in inequalities~\eqref{eq:packing:bounds} and~\eqref{eq:thickCI:bounds}. 
In~\refS{sec:nibble} we shall establish these upper bounds using the following 
proof~strategy, which finds a clique partition~$\cP$ of~$\Gnp$ with the desired properties, 
i.e., size and chromatic index bounds.

\textbf{Step~1: Decomposing the edges of~$\Gnp$ into a clique partition~$\cP$.} 
We first use a semi-random `nibble-type' algorithm to incrementally construct a decreasing sequence of $n$-vertex~graphs  
\begin{equation}
\label{eq:sketch:Gi}
\Gnp=G_0 \: \supseteq \: G_1 \: \supseteq \; \cdots \; \supseteq \: G_I ,
\end{equation}
inspired by the semi-random approaches of Frieze and Reed~\cite{FR1995} and Guo and Warnke~\cite{GW2020}. 
Omitting some technicalities, the main idea is to obtain~$G_{i+1}$ from~$G_i$ 
by removing the edges of a random collection~$\cK_i$ of cliques of size~$k_i=O(\log n)$ from~$G_i$. 
We then put all remaining edges of $G_I$ into~$\cK_I$, 
to ensure~that 
\begin{equation}
\label{eq:sketch:cP}
\cP \; = \; \cK_0 \cup \: \cdots \: \cup \cK_I
\end{equation}
covers all edges of~$\Gnp$. 
Here we exploit the flexibility of the semi-random approach, which allows us to add extra wrinkles to the algorithm. 
In particular, using concentration inequalities, these extra wrinkles enable us to 
show that whp all graphs~$G_i$ stay pseudo-random, 
i.e., that~$G_i$ `looks like' a random graph~$\Gnpp{p_i}$ with suitably decaying edge-probabilities~$p_i$;  
see~\refS{sec:alg:details} and~\refT{thm:pseudorandom} for the~details.

\textbf{Step~2: Coloring the clique partition~$\cP$.} 
We then use the basic observation 
\begin{equation}
\label{eq:sketch:cP:col1}
\ci(\cP) \; \le \; \sum_{0 \le i \le I}\ci(\cK_i) \; \le \; \sum_{0 \le i <  I}\ci(\cK_i) \: + \: 2\Delta(G_I), 
\end{equation}
where the last inequality~$\ci(\cK_i) \le  2\Delta(G_I)$ follows from Vizing's theorem, since~$\cK_I=E(G_I)$ simply contains all edges of~$G_I$. 
Thinking of~$\cK_i$ as a hypergraph with vertex-set~$V(G_i)$ and~edge-set~$\cK_i$, 
we would like to similarly bound~$\ci(\cK_i) = O(\Delta(\cK_i))$, but there is a major obstacle here. 
Namely, as discussed, such Pippenger--Spencer type coloring results only apply to hypergraphs with edges of constant size, 
and their proofs are hard to extend to hypergraphs with edges of size~$O(\log n)$ such as~$\cK_i$.  
We overcome this technical obstacle by exploiting that~$\cK_i$ is 
a random collection of cliques from~$G_i$. 
Crucially, this enables us to bound~$\ci(\cK_i)$ using 
our new probabilistic Pippenger--Spencer type result \refT{thm:ps:random}, 
which efficiently colors such random hypergraphs with large edges. 
In view of~\eqref{eq:ps:random:case}, it thus becomes plausible that~whp
\begin{equation}
\label{eq:sketch:cP:col2}
\ci(\cP) \; \le \; \sum_{0 \le i < I}O\bigpar{\Delta(\cK_i)} \: + \: O\bigpar{\Delta(G_I)} ,
\end{equation}
where the pseudo-random properties are key for verifying the technical assumptions 
of~\refT{thm:ps:random}. 
Using again pseudo-randomness to estimate~$\Delta(\cK_i)$ and~$\Delta(G_I)$, 
it turns out\footnote{Heuristically, the form of the upper bound~\eqref{eq:sketch:cP:col3} can be motivated as follows:
 \eqref{eq:sketch:Gi} and~$G_i \approx \Gnpp{p_i}$ loosely suggest $\ccc(\Gnp) \le \sum_{0 \le i \le I}\ccc(\Gnpp{p_i})$, 
which together with~\eqref{eq:thickCI:bounds} and~$\ccc(\Gnpp{p_I}) \le 2\Delta(\Gnpp{p_I}) = O(np_I)$ makes~\eqref{eq:sketch:cP:col3} a natural target~bound.}
that whp
\begin{equation}
\label{eq:sketch:cP:col3}
\ci(\cP) \; \le \; \sum_{0 \le i < I}O\biggpar{\frac{np_i}{\log_{1/p_i}n}} \: + \: O\bigpar{np_I} \; \le \; \cdots \; \le \; O\biggpar{\frac{np}{\log_{1/p} n}} ,
\end{equation}
where the exponentially decreasing edge-probabilities~$p_i$ 
will ensure that in estimate~\eqref{eq:sketch:cP:col3} the bulk of the contribution comes 
from the case~$i=0$ with~$p_0=p$;  
see~\refS{sec:packing:CI} for the details.
Finally, the whp size estimate~$|\cP|=O\bigpar{n^2p/(\log_{1/p}n)^2}$ 
can be obtained in a similar but simpler way; see~\refS{sec:Psize}. 


\subsubsection{Technical result: weakly pseudo-random clique partition} 
As we shall see in~\refS{sec:nibble}, the outlined proof strategy gives the following technical result 
for large edge-probabilities~$p=p(n)$.  
\refT{thm:mainpacking} intuitively guarantees that the random graph~$\Gnp$ has 
a weakly pseudo-random clique partition~$\cP$, 
i.e., which simultaneously has small size, thickness and chromatic index. 
\begin{theorem}\label{thm:mainpacking}
There is a constant~$\alpha>0$ so that, for every fixed real~${\gamma \in (0,1)}$, 
there are constants~${B,C>0}$ such that the following holds. 
If the edge-probability~${p=p(n)}$ satisfies~${n^{-\alpha} \le p \le 1-\gamma}$, 
then whp there exists a clique partition~$\cP$ of the edges of~$\Gnp$ satisfying 
${\max_{K \in {\cP}}|K| \le \log_{1/p} n}$, \linebreak[3]
${|\cP| \le B n^2p/(\log_{1/p} n)^2}$ and 
${\Delta(\cP) \le \ci(\cP) \le C np/\log_{1/p}n}$. 
\end{theorem}
\begin{remark}
The proof shows that the whp conclusion in fact holds with probability at least~$1-n^{-\omega(1)}$. 
\end{remark}
\noindent
After potentially increasing the constants~$B,C>0$, 
this theorem readily implies the upper bounds in~\eqref{eq:packing:bounds} and~\eqref{eq:thickCI:bounds} of Theorems~\ref{thm:packing}--\ref{thm:thickCI}, 
since for smaller edge-probabilities~$p=p(n) \le n^{-\alpha}$ the trivial clique partition ${\cP:=E(\Gnp)}$ consisting of all edges of~$\Gnp$ 
easily\footnote{Using well-known estimates, it is easy to see that whp~$|\cP| = |E(\Gnp)|\sim \binom{n}{2}p < 2\alpha^{-2} \cdot \binom{n}{2}p/(\log_{1/p}n)^2$ for~$n^{-2} \ll p \le n^{-\alpha}$ and whp~$\ci(\cP) = \ci(\Gnp) \le \Delta(\Gnp)+1 \sim np < 2\alpha^{-1} \cdot np/(\log_{1/p}n)$ for~$n^{-1}\log n \ll p \le n^{-\alpha}$.} 
gives the desired bounds due to~$1 \le 1/(\alpha \log_{1/p} n)$.

\subsection{Organization}
%
In~\refS{sec:nibble} we prove our main technical clique partition result \refT{thm:mainpacking} 
(which as discussed implies Theorems~\ref{thm:prague}, \ref{thm:packing} and~\ref{thm:thickCI}), 
by analyzing a semi-random greedy clique partition algorithm 
using concentration inequalities and our new chromatic index result \refT{thm:ps:random}. 
We then prove \refT{thm:ps:random} in~\refS{sec:edgechr}, 
by analyzing a natural random greedy edge coloring algorithm using the differential equation method. 
The final \refS{sec:concl} discusses some open problems, 
sharpens the lower bounds of Theorems~\ref{thm:packing}--\ref{thm:thickCI} for constant edge-probabilities~$p$, 
and also records strengthenings of Theorems~\ref{thm:packing}--\ref{thm:thickCI} for many small edge-probabilities~${p=p(n) \to 0}$.

\section{Semi-random greedy clique partition algorithm}\label{sec:nibble}
In this section we prove \refT{thm:mainpacking} (and thus Theorems~\ref{thm:prague}, \ref{thm:packing} and~\ref{thm:thickCI}, see Sections~\ref{sec:cp}--\ref{sec:strategy}) 
by showing that a certain semi-random greedy algorithm is likely to find the desired clique partition~$\cP$ of the binomial random graph~$\Gnp$. 
This algorithm iteratively adds cliques to~$\cP$, and the main idea is as follows. 
Writing $G_i \subseteq \Gnp$ for the subgraph containing all edges of~$\Gnp$ 
which are edge-disjoint from the cliques added to~$\cP$ during the first~$i$~iterations, 
we randomly sample a collection~$\Gamma_i$ of cliques from~$G_i$ (of suitable size~$k_i$).   
We then alter this collection to ensure that there are no edge-overlaps between the cliques, 
and add the resulting edge-disjoint collection~$\Gamma_i^* \cup D_i$ of cliques to~$\cP$. 
Finally, after a sufficiently large number of~$I$ iterations, 
we add all remaining so-far uncovered edges of~$G_I \subseteq \Gnp$ to~$\cP$ (as cliques of size~two). 

In fact, we shall use  an additional wrinkle for technical reasons: 
in each iteration of the algorithm we add an extra set~$S_i$ of random edges to~$\cP$, 
which helps us to ensure that the graphs~$G_i=([n],E_i)$ stay pseudo-random, 
i.e., resemble a random graph~$\Gnpp{p_i}$ with suitable decaying edge-probabilities~$p_i$. 

\subsection{Details of the semi-random `nibble' algorithm}\label{sec:alg:details}
Turning to the technical details of our clique partition algorithm, let 
\begin{align}\label{eq:defs}
    k := \bigceil{\ca\log_{1/p} n}
, \quad 
I := \bigceil{\cc k^\cc \log k}
, \quad 
p_i := p e^{-i/k^\cc}
, \quad 
k_i := \bigceil{\ca\log_{1/p_i} n}
, \quad 
\eps := n^{-\beps} ,
\end{align}
where we fix the absolute constants~$\ca:=1/9$ and~$\cc:=9$ for concreteness 
(we have made no attempt to optimize these constants, and the reader looses little by simply assuming that $\ca$ and~$\cc$ are always sufficiently small and large, respectively, whenever needed). 
For any vertex-subset~$S \subseteq [n]$ with~$|S| \le j$ we define
\begin{align}\label{eq:Cji}
    \cC_{S, j, i} := \lrset{J \subseteq V \:\colon\: S \subseteq J,\; \abs{J} = j,\;  \tbinom{J}{2} \setminus \tbinom{S}{2} \subseteq E_i}.
\end{align}
In words, if~$S$ forms a clique in the graph~$G_i=([n],E_i)$, then $\cC_{S, j, i}$ corresponds to the set of all $j$-vertex cliques of $G_i$ that contain~$S$. 
Furthermore, if~$G_i$ indeed heuristically resembles the random graph~$\Gnpp{p_i}$ (as suggested above, and later made precise by \refT{thm:pseudorandom}), 
then we expect that~$|\cC_{S, j, i}| \approx \mu_{|S|,j,i}$,~where 
\begin{align}\label{eq:musji}
    \mu_{s,j,i} := \binom{n-s}{j-s} p_i^{\binom{j}{2} - \binom{s}{2}}.
\end{align}
Writing~$E(\cC) := \bigcup_{K \in \cC} E(K)$ for the edges covered by a family~$\cC$ of cliques, 
after defining 
\begin{align}\label{def:qi:zeta}
q_i := \frac{1}{(1 + \eps) k^\cc \mu_{2, k_i, i} } 
\quad \text { and } \quad  
\zeta_{e,i} := 1 - (1-q_i)^{\max\set{(1+\eps)\mu_{2,k_i,i} - |\cC_{e, k_i, i}|,\; 0}} 
\end{align}
we now formally state the algorithm that finds the desired clique partition~$\cP$ of~$\Gnp$.  
%
\begin{algorithm}[H]\caption{{\bfseries Algorithm: Semi-random greedy clique partition}}\label{alg:semi} 
\begin{algorithmic}[1]
\STATE Set~$\cP_0 := \emptyset$ and~$G_0 := ([n], E_0)$, where~$E_0 := E(\Gnp)$.
\FOR{$i = 0$ \textbf{to} $I-1$}
		\STATE Let~$\cC_i:=\cC_{\emptyset, k_i, i}$ contain all $k_i$-vertex cliques of~$G_i$. 
    \STATE Generate~$\Gamma_i \subseteq \cC_i$: independently include each clique~$K \in \cC_i$  with probability~$q_i$. 
    \STATE Generate~$S_i \subseteq E_i$: independently include each edge~$e \in E_i$  with probability~$\zeta_{e,i}$. 
      \STATE Let~$\Gamma_i^*$ be a size-maximal collection of edge-disjoint $k_i$-vertex cliques in $\Gamma_i$.
      \STATE Set~$\cP_{i+1} := \cP_i \cup \Gamma_i^* \cup D_i \cup \bigpar{S_i \setminus E(\Gamma_i)}$, where~$D_i := E(\Gamma_i) \setminus E(\Gamma_i^*)$.
    \STATE Set~$G_{i+1} := ([n], E_{i+1})$, where~$E_{i+1} := E_i \setminus \bigpar{E(\Gamma_i) \cup S_i}$.
\ENDFOR
\STATE Return~$\cP := \cP_{I} \cup E_I$.
\end{algorithmic}
\end{algorithm}%

One may heuristically motivate the technical definitions~\eqref{def:qi:zeta} of~$q_i$ and~$\zeta_{e,i}$ as follows. 
The `inclusion' probability~$q_i$ will intuitively ensure that, for any fixed edge~$e \in E_i$, the expected number of cliques in~$\Gamma_i$ containing~$e$ is roughly~$|\cC_{e, k_i, i}| \cdot q_i \approx \mu_{2, k_i, i} \cdot q_i \approx 1/k^{\cc}$. 
This makes it plausible that the cliques in~$\Gamma_i$ are largely edge-disjoint, i.e., that~$|\Gamma_i^*| \approx |\Gamma_i|$. 
The `stabilization' probability~$\zeta_{e,i}$ will intuitively ensure~that 
\[
\P(e \in E_{i+1} \mid e \in E_i) = (1 - q_i)^{|\cC_{e, k_i, i}|} \cdot (1 - \zeta_{e,i}) 
\approx (1 - q_i)^{(1 + \eps) \mu_{2, k_i, i}} 
\approx e^{-1/k^\cc} .
\] 
Since all edges~$e \in E_i$ of~$G_i$ have roughly the same probability of appearing in~$E_{i+1}$, 
it then inductively becomes plausible that~$G_{i+1}$ resembles 
a random graph~$\Gnpp{p_{i+1}}$ with edge-probability~$p_{i+1} \approx p_i \cdot e^{-1/k^\cc}$. 

\subsection{The clique partition~$\cP$: proof of \refT{thm:mainpacking}}
In this section we prove \refT{thm:mainpacking} by analyzing the clique partition~$\cP$ produced by the semi-random greedy algorithm. 
Recalling the definitions~\eqref{eq:Cji}--\eqref{eq:musji} of~$|\cC_{S, j, i}|$ and~$\mu_{s, j, i}$, 
\refT{thm:pseudorandom} confirms our heuristic that~$G_i$ stays pseudo-random, 
i.e., resembles the random graph~$\Gnpp{p_i}$ with respect to various clique statistics. 
%
\begin{theorem}[Pseudo-randomness of the graphs~$G_i$]\label{thm:pseudorandom}
Let~$p=p(n)$ satisfy~$n^{-\ca/\cc} \le p \le 1-\gamma$, where~$\gamma \in (0,1)$ is a constant.
Then, with probability at least $1 - n^{-\omega(1)}$, for all $0 \leq i \leq I$ the following event $\cX_i$ holds: for all~$S \subseteq V$ and $j$ with $0 \leq |S| \leq j \leq k_i$, we have 
\begin{equation}\label{eq:XSji}
    |\cC_{S, j, i}| \; = \; (1 \pm \eps) \cdot \mu_{|S|, j, i}.
\end{equation}
\end{theorem}
We defer the proof of this technical auxiliary result to \refS{sec:mainnibble}, 
and first use it (together with our new edge-coloring result \refT{thm:ps:random}) to prove \refT{thm:mainpacking} with~$\alpha := \ca/\cc$. 
To this end, we henceforth tacitly assume~$n^{-\ca/\cc}\le p \le 1-\gamma$. 
In particular, for~$0 \le i \le I$ it then is routine to check that 
\begin{equation}\label{eq:pi:ki}
8 < \tau-o(1) \le \frac{\ca \log n}{\log(k^{2\cc}/p)} \le k_i \le k \le n^{o(1)} 
\qquad \text{ and } \qquad 
\min_{0 \le s \le k_i-1} p_i^s \ge p_i^{k_i-1} \ge n^{-\ca} .
\end{equation}
By construction of~$\cP$, it also follows that~$\max_{K \in \cP}|K| \le k \le \log_{1/p}n$. 
To complete the proof of~\refT{thm:mainpacking}, 
it thus remains to bound the size and chromatic index of the clique partition~$\cP$, respectively.

\subsubsection{Chromatic index of~$\cP$}\label{sec:packing:CI} 
We first focus on the chromatic index of the clique partition~$\cP$, 
which is easily seen to be 
(by separately coloring different subsets of the cliques, using disjoint sets of colors)  
at most 
\begin{equation}
\label{eq:P:cbound:1}
\ci(\cP) \leq \sum_{0 \leq i \leq I-1} \Bigpar{\ci(\Gamma_i) + \ci(D_i) + \ci(S_i)} + \ci(E_I).
\end{equation}
For~$\cS \in \{\Gamma_i,D_i,S_i,E_I\}$, let $\cS^{(v)} \subseteq \cS$ denote the subset of cliques that contain the vertex $v$.
Since the cliques in $D_i, S_i,E_I$ are all simply edges, 
using Vizing's theorem it follows that 
\begin{equation}
\label{eq:P:cbound:2}
\ci(\cP) \leq \sum_{0 \leq i \leq I-1} \Bigpar{\ci(\Gamma_i) + \max_{v \in [n]}\bigabs{D^{(v)}_i} + \max_{v \in [n]}\bigabs{S^{(v)}_i}} + \max_{v \in [n]}\bigabs{E^{(v)}_I} + (2I+1).
\end{equation}

In the following we bound the contributions of each of these terms, 
and we start with the main term~$\ci(\Gamma_i)$. 
Gearing up to apply our new Pippenger--Spencer type chromatic index result \refT{thm:ps:random} to the random set~$\Gamma_i \subseteq \cC_i$ of cliques, 
let ${\cH := ([n], \cC_i)}$ denote the \mbox{$k_i$-uniform} auxiliary hypergraph consisting of all \mbox{$k_i$-vertex} cliques in~$G_i$. 
Note that~$\Gamma_i$ has the same distribution as the edge-set of~$\cH_{q_i}$,  
where the random subhypergraph~$\cH_{q_i} \subseteq \cH$ is defined as in \refC{c:ps:random} with $r=k_i$ and $q=q_i$ 
(we defer the proof of \refC{c:ps:random} to \refApp{apx:lower}, since this standard reduction to \refT{thm:ps:random} is rather tangential to the main argument~here). 
%
\begin{corollary}[Convenient variant of \refT{thm:ps:random}]\label{c:ps:random}
There is~${\xi=\xi(\delta)>0}$ such that if the assumptions of \refT{thm:ps:random} hold for a given~$n$-vertex $r$-uniform hypergraph~$\cH$, 
with assumption~${m \le n^{r n^{\sigma/5}}}$ replaced by~${m \le\xi e(\cH)}$,   
then we have $\Pr(\ci(\cH_q) \le {(1+2\delta)r m/n}) \ge 1-n^{-\omega(r)}$, 
where~$\cH_q$ denotes the random subhypergraph of~$\cH$ where each edge~$e \in \cH$ is independently included with probability~${q:=m/|E(\cH)|}$. 
\end{corollary}
\noindent
Conditional on~$\cX_i$, we will apply this corollary to~$\cH=([n], \cC_i)$ 
with $r:=k_i$, $m:=|E(\cH)|q_i$, $D:=\mu_{1,k_i,i}$, $q:=q_i$, as well as 
\begin{equation}\label{def:bdelta}
b:=2\ca/\log(1/(1-\gamma)) \quad \text{ and } \quad \delta:=30b/\ca. 
\end{equation}
We now verify the technical assumptions of \refC{c:ps:random} (and thus~\refT{thm:ps:random}). 
Using the definition~\eqref{eq:defs} of~$k_i$ and inequality~\eqref{eq:pi:ki} together with $p_i \le p \le 1-\gamma$, 
we obtain $2 < k_i \le 2\ca (\log n)/\log(1/p_i) \le b \log n$. 
Using the estimate~\eqref{eq:XSji} of~$\cX_i$ together with the definition~\eqref{def:qi:zeta} of $q_i$, it follows that
\begin{equation}\label{eq:m}
m = |\cC_{\emptyset, k_i, i}| \cdot q_i 
= \frac{(1\pm \eps) \mu_{0,k_i,i}}{(1+\eps)k^\cc \mu_{2,k_i,i}} 
=  \frac{1\pm \eps}{1+\eps} \cdot \frac{n(n-1)p_i}{k_i(k_i-1)k^\cc} ,
\end{equation}
so that~$m \ge n^{2-\ca-o(1)} \gg n^{1+\ca}$ by~\eqref{eq:pi:ki} and choice of $\ca$. 
Recalling that $\eps=n^{-\ca}$, estimate~\eqref{eq:XSji} implies that~$\cH=([n], \cC_i)$ satisfies the degree condition in~\eqref{as:degcodeg}. 
We also have $\mu_{2,k_i,i}/\mu_{1,k_i,i} \le \bigpar{\Omega(n/k_i) \cdot p_i}^{-1} \le n^{-1+\ca+o(1)} \ll n^{\ca}$, 
which in view of~\eqref{eq:XSji} and~$D=\mu_{1,k_i,i}$ implies that~$\cH$ also satisfies the codegree condition in~\eqref{as:degcodeg}. 
We similarly infer~$D = \bigpar{\Omega(n/k_i) \cdot p_i^{k_i/2}}^{k_i-1} \ge (n^{1-\ca-o(1)})^{4} \gg n^3$, 
so that $m = O(n^2/k_i) \ll D/r \ll e(\cH)$. 
We thus may apply \refC{c:ps:random} to~$\cH$, 
which together with our above discussion gives
\begin{equation}
\label{eq:P:cbound:Gammai}
\Pr\bigpar{{\ci(\Gamma_i) \ge {(1+2\delta)k_i m/n}} \mid \cX_i} = \Pr\bigpar{{\ci(\cH_q) \ge {(1+2\delta)k_i m/n}} \mid \cX_i} \le n^{-\omega(1)} .
\end{equation}

In the following we fix a vertex~$v \in [n]$, and bound~$|S^{(v)}_i|$ and $|D^{(v)}_i|$ separately. 
Using~\eqref{eq:XSji} together with ${1-(1-q_i)^{2\eps\mu_{2,k_i,i}}} \le {2\eps\mu_{2,k_i,i}q_i}  \le {2 \eps k^{-\cc}}$ and $\eps = n^{-\ca} \ll k^{-2} \le k_i^{-2}$, 
it follows that 
\begin{equation}\label{eq:svi}
\E\bigpar{|S_i^{(v)}| \: \mid \cX_i} 
\le |\cC_{\{v\},2,i}| \cdot \bigpar{1-(1 - q_i)^{2\eps\mu_{2, k_i, i}}} 
\le  \frac{2\eps n p_i}{k^\cc} \ll \frac{np_i}{k_i^2 k^{\cc}} =: \lambda.
\end{equation}
Note that $\lambda \geq n^{1 - \ca - o(1)} \gg \log n$ by inequality~\eqref{eq:pi:ki} and choice of~$\ca$. 
Furthermore, since~$|S_i^{(v)}|$ is a sum of independent indicator random variables, 
standard Chernoff bounds (such as~\cite[Theorem~2.1]{JLR}) imply 
\begin{equation}\label{eq:degSi:P}
\Pr\bigpar{|S_i^{(v)}| \ge 2\lambda \mid \cX_i} \le \exp\bigpar{-\Theta(\lambda)} \le n^{-\omega(1)} .
\end{equation}
Turning to~$|D^{(v)}_i|$, let $X$ denote the number of unordered pairs ${\{K', K''\} \in \binom{\Gamma_i}{2}}$ with ${|\{K', K''\} \cap \Gamma_i^{(v)}| \ge 1}$ and~${|E(K') \cap E(K'')| \geq 1}$. 
Since each of these edge-overlapping clique pairs contributes at most~$k_i \le k$ edges to~$|D_i^{(v)}|$, we infer $|D^{(v)}_i| \leq k X$. 
Furthermore, using~\eqref{eq:XSji} and~\eqref{def:qi:zeta}, it follows similarly to~\eqref{eq:m} that
\begin{equation}\label{eq:degDi:E}
\E(X \mid \cX_i) 
\le \sum_{K' \in \cC_{\{v\}, k_i, i}} \sum_{e \in \binom{K'}{2}} \sum_{K'' \in \cC_{e, k_i, i}} q_i^2
\le |\cC_{\{v\}, k_i, i}| \cdot \binom{k_i}{2} \cdot (1 + \eps)\mu_{2, k_i, i} \cdot q_i^2 \le \frac{k_i n p_i}{k^{2\cc}} =: \mu .
\end{equation}
Conditioning on the event~$\cX_i$, 
we shall bound~$X$ using the following upper tail inequality for combinatorial random variables, 
which is a convenient corollary of~\cite[Theorem~9]{Warnke2017}.  
\begin{lemma}\label{lem:UT}
Let $(\xi_j)_{j\in\Lambda}$ be a finite family of independent random variables with $\xi_j\in \{0,1 \}$. 
Let $(Y_\alpha)_{\alpha\in\mathcal{I}}$ be a finite family of random variables with $Y_\alpha := \indic{\xi_j=1\text{ for all } j\in \alpha}$. 
Defining $\mathcal{I}^+:= \{\alpha \in \mathcal{I}: Y_{\alpha}=1\}$, let~$\mathcal{G}$ be an event that implies $\max_{\alpha \in \mathcal{I}^+}|\{ \beta\in\mathcal{I}^+: \beta\cap\alpha\neq\emptyset\} | \leq C$. 
Set~$X:=\sum_{\alpha\in \mathcal{I}}Y_\alpha$, and assume that~$\E X \le \mu$. 
Then, for all~$x>\mu$,
\begin{equation}\label{eq:C}
\P(X\geq x \text{ and } \mathcal{G} ) 
\: \le \: 
\bigpar{e\mu/x}^{x/C} .
\end{equation}
\end{lemma}
\noindent
We will apply \refL{lem:UT} to~$X$ with~$\Lambda =\cC_i$, the independent random variables~$\xi_K := \indic{K \in \Gamma_i}$, and $\cI$ equal to the set of unordered pairs~${\{K', K''\} \in \binom{\cC_i}{2}}$ with ${|\{K', K''\} \cap \cC_{\{v\}, k_i, i}| \ge 1}$ and~${|E(K') \cap E(K'')| \geq 1}$. 
Let~$\cG$ denote that the event that each edge~$e \in E_i$ is contained in at most~$z := \ceil{\log n}$ cliques in~$\Gamma_i$. 
Clearly, $\cG$ implies that each clique~$K' \in \Gamma_i$ has edge-overlaps with a total of at most $\binom{k_i}{2} \cdot z $  cliques $K'' \in \Gamma_i$, 
so that the parameter~$C:= 2 \cdot \binom{k_i}{2}z \le k^2z$ works in \refL{lem:UT}. 
Recalling~${|D_i^{(v)}| \leq k X}$, by invoking inequality~\eqref{eq:C} with ${x := \lambda/k} \ge {k^{\tau-4}\mu} > {e^2 \mu}$ it follows~that 
\begin{equation}\label{eq:degDi:P:1}
\P\bigpar{|D_i^{(v)}| \geq \lambda \text{ and } \mathcal{G} \mid \cX_i}  
\leq \P\bigpar{X \geq \lambda/k \text{ and } \mathcal{G} \mid \cX_i}  
\leq \exp\bigpar{-\Theta(\lambda/(k^3z))} \le n^{-\omega(1)},
\end{equation}
where the last inequality uses~$\lambda/(k^3z) \ge \lambda n^{-o(1)} \gg \log n$ analogous to~\eqref{eq:degSi:P}.
With an eye on the event~$\cG$, note that conditional on $\cX_i$ we have~$|\cC_{e, k_i, i}| q_i \le  k^{-\cc} \le 1$ for each edge~$e \in E_i$. 
Recalling~$z =\ceil{\log n}$, 
by taking a union bound over all edges~$e \in E_i$ it now is routine to see that 
\begin{equation}\label{eq:degDi:P:2}
\P(\neg\mathcal{G} \mid \cX_i) \le \sum_{e \in E_i} \binom{|\cC_{e, k_i, i}|}{z} q_i^z \le |E_i| \cdot \Bigpar{|\cC_{e, k_i, i}|q_ie/z}^z \le n^2 \cdot (e/z)^z \le n^{-\omega(1)} .
\end{equation}

To sum up, by combining the above inequalities~\eqref{eq:P:cbound:Gammai}, \eqref{eq:degSi:P}, and \eqref{eq:degDi:P:1}--\eqref{eq:degDi:P:2} for~$0 \le i \le I-1$ with the degree estimate~$|E^{(v)}_{I}|=|\cC_{\{v\},k_I,I}| = (1\pm\eps)(n-1)p_I$ from~\eqref{eq:XSji}, 
using $I=n^{o(1)}$ and \refT{thm:pseudorandom} it follows (by a standard union bound argument) that the chromatic index~\eqref{eq:P:cbound:2} of~$\cP$ is whp at~most 
\begin{equation}
\label{eq:cindex:bound:1}
\ci(\cP) \leq \sum_{0 \leq i \leq I-1} \biggpar{\frac{(1+2\delta)2np_i}{k_ik^\cc} + \frac{3np_i}{k_i^2k^\cc}} \: + \: (1+\eps)n p_I + n^{o(1)},
\end{equation}
where the~$k_i^2 > k_i$ term will be useful in \refS{sec:Psize}. 
Let ${\pi := \log(1/p)}$ and ${f(x) := e^{-x}(1+x/\pi)}$. 
Using ${p_i=p \cdot e^{-i/k^\cc}}$ and $k_i \ge {\ca \log_{1/p}(n)/(1+i/(k^\cc \pi))}$ 
as well as $n^{\ca} \ll n^{1-\ca} \le n p_I \le n p/k^{\cc}$, 
it follows that 
\begin{equation}
\label{eq:cindex:bound:2}
\ci(\cP) \leq \frac{(5+4\delta)np}{\ca \log_{1/p}n} \sum_{0 \leq i \leq I-1} \frac{f(i/k^\cc)}{k^\cc} \: + \: \frac{3np}{\bigpar{\ca \log_{1/p}n} k^{\cc-1}}.
\end{equation}
On~$[0,\infty)$ the function $f(x)$ first increases and then decreases, with a maximum at $x^* := {\max\{0, 1-\pi\}}$. 
By comparing the sum with an integral, it then is standard to see that 
\begin{equation}
\label{eq:sum-cindex:bound}
\sum_{0 \leq i \leq I-1} \frac{f(i/k^\cc)}{k^\cc} 
\le \int_{0}^{\infty}f(x)dx + 2 f(x^*)/k^\cc 
\: \le \: 1 + O\bigpar{\pi^{-1} + k^{-\cc}}.
\end{equation}
Combining inequalities~\eqref{eq:cindex:bound:2}--\eqref{eq:sum-cindex:bound} with the definition~\eqref{def:bdelta} of~$\delta$, 
after noting~$\pi \ge \log(1/(1-\gamma))> 0$ and $\min\{k^{\cc},k^{\cc-1}\} > 1$ it follows that 
there is a constant~$C=C(\ca,\gamma) > 0$ such that whp~$\ci(\cP)  \le {C np/\log_{1/p}n}$. 

\subsubsection{Size of~$\cP$}\label{sec:Psize} 
It remains to bound the size of the clique partition~$\cP$, which by construction is at most
\begin{equation}
\label{eq:P:bound:1}
|\cP| \leq \sum_{0 \leq i \leq I-1} \Bigpar{|\Gamma_i| + |D_i| + |S_i|} + |E_I|. 
\end{equation}
Rather than estimating each of these terms (which is conceptually straightforward), 
we shall instead reuse known estimates from \refS{sec:packing:CI}. 
A routine double-counting argument gives ${|\Gamma_i| \cdot k_i} \leq {\sum_{K \in \Gamma_i}|K|} = {\sum_{v \in [n]}|\Gamma_i^{(v)}|} \leq {n \cdot \ci(\Gamma_i)}$.
Recalling that $D_i, S_i,E_I$ are simply sets of edges, it follows~that 
\begin{equation}
\label{eq:P:bound:2}
|\cP| \leq \sum_{0 \leq i \leq I-1} \Bigpar{n/k_i \cdot \ci(\Gamma_i) + n \cdot \max_{v \in [n]}\bigabs{D^{(v)}_i} + n \cdot \max_{v \in [n]}\bigabs{S^{(v)}_i}} + n \cdot \max_{v \in [n]}\bigabs{E^{(v)}_I}.
\end{equation}
After comparing the above upper bound for~$|\cP|$ with~\eqref{eq:P:cbound:2}, 
we see that the proof of~\eqref{eq:cindex:bound:1} implies the following estimate:
the size~\eqref{eq:P:bound:1} of~$\cP$ is whp at~most 
\begin{equation}
\label{eq:P:bound:2}
|\cP| \leq \sum_{0 \leq i \leq I-1} \biggpar{\frac{(1+2\delta)2n^2p_i}{k_i^2k^\cc} + \frac{3n^2p_i}{k_i^2k^\cc}} \: + \: (1+\eps)n^2 p_I .
\end{equation}
Recalling~$\pi = \log(1/p)$, set~$g(x) := e^{-x}(1 + x/\pi)^2$. 
Proceeding similarly to~\eqref{eq:cindex:bound:1}--\eqref{eq:sum-cindex:bound}, 
using~${\int_{0}^{\infty}g(x)dx} = {1 + O(\pi^{-1} + \pi^{-2})}$ 
it follows that there is a constant~$B=B(\ca,\gamma) > 0$ such that whp
\begin{equation}
\label{eq:P:bound:3}
|\cP| \leq \frac{(5 + 4\delta)n^2p}{\bigpar{\ca \log_{1/p}n}^2} \sum_{0 \leq i \leq I-1} \frac{g(i/k^\cc)}{k^\cc} \: + \: \frac{2n^2p}{\bigpar{\ca \log_{1/p}n}^2 k^{\cc-2}} \; \le \; B \frac{n^2p}{(\log_{1/p}n)^2},
\end{equation}
which completes the proof \refT{thm:mainpacking} (modulo the deferred proof of \refT{thm:pseudorandom}). 
\noproof

\subsection{Pseudo-randomness of the graphs~$G_i$: proof of \refT{thm:pseudorandom}}\label{sec:mainnibble}
In this section we give the deferred proof of \refT{thm:pseudorandom}.
For any vertex-subset $S \subseteq [n]$ we define
\begin{equation}\label{def:NSi}
N_{S,i} := \bigabs{\bigset{w \in V \setminus S \: \colon \: S \times \{w\} \subseteq E_i}} . 
\end{equation}
In words, $N_{S,i}$ denotes the number of common neighbors of~$S$ in $G_i=([n],E_i)$. 
Recalling that~$G_i$ heuristically resembles the random graph~$\Gnpp{p_i}$, we expect that~$N_{S,i} \approx (n-|S|)p_i^{|S|}$; 
so to avoid clutter we set 
\begin{equation}\label{def:lambdasi}
\lambda_{s,i} :=  (n - s) p_i^{s}. 
\end{equation}
The following pseudo-random result establishes~\refT{thm:pseudorandom} by confirming this heuristic prediction. 
\begin{theorem}[Strengthening of \refT{thm:pseudorandom}]\label{thm:mainnibble}
Let~$p=p(n)$ satisfy~$n^{-\ca/\cc} \le p \le 1-\gamma$, where~$\gamma \in (0,1)$ is a constant. 
Then, with probability at least $1 - n^{-\omega(1)}$, for all $0 \leq i \leq I$ the following event~$\cN_i$ holds: 
for all~$S \subseteq [n]$ with~$0 \leq |S| \le k_i-1$, 
\begin{equation}\label{eq:mainnibble}
    N_{S,i} \; = \; \bigpar{1 \pm (i+1) \deps } \cdot \lambda_{|S|,i}.
\end{equation}
Furthermore, 
$\cN_i$ implies the event~$\cX_i$ from \refT{thm:pseudorandom}
for~$0 \le i \le I$ and~$n \ge n_0(\ca,\cc)$. 
\end{theorem}
\begin{proof}
Noting~$k I \deps \le n^{o(1)-\beps}\eps \ll \eps$ it is routine to see that $\cN_i$ implies $\cX_i$, but we include the proof for completeness.
Fixing~$|S| \le j \le k_i$ and $0 \le i \le I$, 
we shall double-count the number of vertex-sequences $x_{|S|+1}, \ldots, x_{j} \in [n] \setminus S$ with the property that $S \cup \{x_{|S|+1}, \ldots, x_{j}\} \in \cC_{S, j , i}$. 
Using~\eqref{eq:mainnibble} to sequentially estimate the number of common neighbors of $S \cup \{x_{|S|+1}, \ldots, x_{\ell}\}$, 
noting~$j \cdot I \deps \le k I \deps \ll \eps$ it follows~that 
\[
(j - |S|)! \cdot |\cC_{S, j, i}| = \hspace{-0.25em}\prod_{|S| \le \ell \le j-1} \hspace{-0.5em}\Bigpar{\bigpar{1 +O(I\deps)} \cdot (n-\ell)p_i^{\ell}} = (1+o(\eps)) \cdot \mu_{|S|,j,i} \cdot (j - |S|)! ,
\]
which readily gives~\eqref{eq:XSji} for~$n \ge n_0(\ca,\cc)$, establishing the claim that~$\cN_i$ implies~$\cX_i$.  
With this implication and~$I = n^{o(1)}$ in mind, the below auxiliary Lemmas~\ref{lem:N0}--\ref{lem:Ni} then complete the proof of \refT{thm:mainnibble}. 
\end{proof}
\begin{lemma}\label{lem:N0}
We have $\P(\neg \cN_0) \leq n^{-\omega(1)}$.
\end{lemma}
\begin{lemma}\label{lem:Ni}
We have $\P(\neg \cN_{i+1} \mid \cN_i) \leq n^{-\omega(1)}$ for all $0 \le i < I$.
\end{lemma} 
\begin{proof}[Proof of \refL{lem:N0}]
Fix~${S \subseteq [n]}$ with~${|S| \le k-1}$, where~$k=k_0$. 
Note that $N_{S, 0}$ has a Binomial distribution with ${\E N_{S,0}} ={ (n-|S|)p^{|S|}} ={\lambda_{|S|,0}}$, where~$p=p_0$. 
Since $\eps^4 \lambda_{|S|,0} = \Theta\bigpar{n^{1-4\ca} p_0^{|S|}} \ge \Omega(n^{1-5\ca}) \gg k \log n$ by inequality~\eqref{eq:pi:ki} and choice of~$\ca$, 
standard Chernoff bounds (such as~\cite[Theorem~2.1]{JLR}) imply that 
\begin{equation}\label{eq:N0:Chern}
\Pr\bigpar{|N_{S,0}-\lambda_{|S|,0}| \ge \deps \lambda_{|S|,0}|} \le  2 \cdot \exp \bigpar{- \Theta \lrpar{\eps^4 \lambda_{|S|,0}}} \le n^{-\omega(k)} ,
\end{equation}
which completes the proof by taking a union bound over all $n^{O(k)}$ choices of the sets~$S$. 
\end{proof}
Conditioning on the event~$\cN_i$, in the proof of \refL{lem:Ni} we shall bound~$N_{S,i+1}$ using 
the following bounded differences inequality for Bernoulli variables, see~\cite[Corollary~1.4]{warnke2016method} and~\cite[Theorem~3.8]{McDiarmid1998}. 
\begin{lemma}\label{thm:BDI}
Let $(\xi_\alpha)_{\alpha \in \cI}$ be a finite family of independent random variables with~$\xi_\alpha \in \{0,1\}$.
Let $f:\{0,1\}^{|\cI|} \to \RR$ be a function, and
assume that there exist numbers $(c_\alpha)_{\alpha \in \cI}$ such that the
following holds for all $z=(z_{\alpha})_{\alpha \in \cI} \in \{0,1\}^{|\cI|}$ and $z'=(z'_{\alpha})_{\alpha \in \cI} \in \{0,1\}^{|\cI|}$:
$|f(z)-f(z')| \le c_{\beta}$ if $z_{\alpha} = z'_{\alpha}$ for all~$\alpha \neq \beta$.
Define $X:= f\bigl((\xi_\alpha)_{\alpha \in \cI}\bigr)$, $V := \sum_{\alpha \in \cI}c_\alpha^2 \P(\xi_{\alpha}=1)$, and $C:=\max_{\alpha \in \cI}c_\alpha$.
Then, for all~$t \ge 0$,
\begin{equation}\label{eq:BDI}
\P(|X - \E X| \ge t) 
\: \le \:
2 \cdot \exp\biggpar{-\frac{t^2}{2(V+C t)}}.
\end{equation}
\end{lemma}
\begin{proof}[Proof of \refL{lem:Ni}]
To avoid clutter, we henceforth omit the conditioning on~$\cN_i$ from our notation. 
%
Fix~${S \subseteq [n]}$ with~${|S| \le k_i-1}$. 
Gearing up to apply~\refL{thm:BDI} to~$N_{S,i+1}$, 
note that the associated parameter~$V$ is given~by
\begin{equation}\label{eq:sigma:NSi}
V = \sum_{K \in \cC_i} c_K^2 \cdot q_i + \sum_{e \in E_i} \hc_e^2 \cdot \zeta_{e, i} ,
\end{equation}
where $c_K$ is an upper bound on how~much $N_{S, i+1}$ can~change if we~alter whether the clique~$K$ is in~$\Gamma_i$ or~not, 
and $\hc_e$ is an upper bound on how~much $N_{S, i+1}$ can~change if we~alter whether the edge~$e$ is in~$S_i$ or~not. 
To estimate $c_K$ and $\hc_e$, note that any edge in $S \times \{w\}$ uniquely determines $w$. 
By definition~\eqref{def:NSi} of $N_{S, i+1}$, it follows that $\hc_e \leq 1$ and $c_K \leq \tbinom{k_i}{2} \leq k^2$. 
In addition, the number of edges $e \in E_i$ with $\hc_e \neq 0$ is at most $N_{S, i} \cdot |S|$.
Similarly, the number of cliques $K \in \cC_i$ with $c_K \neq 0$ is at most $N_{S, i} \cdot |S| \cdot \max_{e} |\cC_{e, k_i, i}| \leq N_{S, i} |S| \cdot (1 + \eps) \mu_{2, k_i, i}$, where we used that~$\cN_i$ implies~$\cX_i$ (as established above) to bound~$|\cC_{e, k_i, i}|$ via~\eqref{eq:XSji}. 
Since~$(1 + \eps) \mu_{2, k_i, i} \cdot q_i = k^{-\cc}$ by definition~\eqref{def:qi:zeta} of~$q_i$,  
and~$\zeta_{e,i} \ll k^{-\tau}$ by the calculation above~\eqref{eq:svi}, 
using~$|S| \le k$ and~$\cc \ge 5$ we infer that 
\begin{align*}
	 V &\leq N_{S,i} |S| \cdot k^{-\cc} \cdot k^4 + N_{S,i} |S| \cdot k^{-\cc} \le 2N_{S,i} \le 4\lambda_{|S|,i} = \Theta(\lambda_{|S|,i+1}) ,
\end{align*}
where we used \eqref{eq:mainnibble} and $i\deps \le I\deps \le n^{o(1)-2\beps} \ll 1$ to bound~$N_{S,i}$. 
Invoking inequality~\eqref{eq:BDI} of \refL{thm:BDI} with~$C=k^2$, noting~$C \deps \le n^{o(1)-2\beps} \ll 1$ it follows that
\begin{equation*}\label{eq:Ni:BDI}
\Pr\bigpar{|N_{S,i+1}-\E N_{S,i+1}| \ge 0.5 \deps \lambda_{|S|,i+1}} \le  2 \cdot \exp\bigpar{- \Theta(\eps^4 \lambda_{|S|,i+1})} \le n^{-\omega(k)} ,
\end{equation*}
where the last estimate is analogous to~\eqref{eq:N0:Chern}.
To complete the proof 
it thus suffices to show that
\begin{equation}\label{eq:Ni:Exp}
\bigabs{\E N_{S,i+1}-\lambda_{|S|,i+1}} \le (i+1.5)\deps \lambda_{|S|,i+1} . 
\end{equation}
Indeed, $\P(\neg \cN_{i+1} \mid \cN_i) \leq n^{-\omega(1)}$ then follows by taking a union bound over all~$n^{O(k)}$ sets~$S$. 

Turning to the remaining proof of~\eqref{eq:Ni:Exp}, note that by construction
\begin{equation}\label{eq:Ni:Exp:1}
    \E N_{S, i+1} = \sum_{\substack{w \in V\setminus S: \\ S \times \{w\} \subseteq E_i}} \P \bigpar{S \times \{w\} \subseteq E_{i+1}}.
\end{equation}
Let~$U := \bigcup_{e \in F} \cC_{e, k_i, i}$. 
Since~$\cN_i$ implies~$\cX_i$ we obtain $(1 + \eps) \mu_{2, k_i, i} \ge |\cC_{e, k_i, i}|$ via \eqref{eq:XSji}, 
so by definition of~$E_{i+1}$ and definition~\eqref{def:qi:zeta} of $\zeta_{e,i}$ it follows that 
\begin{equation}\label{eq:Ni:Exp:2}
\begin{split}
		\P\lrpar{S \times \{w\} \subseteq E_{i+1}} 
		&= (1-q_i)^{|U|} \cdot \hspace{-0.75em}\prod_{e \in S \times \{w\}} \hspace{-0.75em}(1 - \zeta_{e, i}) \\
		& = (1-q_i)^{|U| - \sum_{e \in S \times \{w\}} |\cC_{e, k_i, i}|} \cdot (1-q_i)^{|S| (1 + \eps) \mu_{2, k_i, i}} .
\end{split}
\end{equation}
Recalling the definition~\eqref{def:qi:zeta} of~$q_i$, 
using estimates~\eqref{eq:XSji}--\eqref{eq:pi:ki} we infer that 
\begin{equation*}
\begin{split}
q_i \cdot \Bigabs{|U| - \hspace{-0.25em}\sum_{e \in S \times \{w\}} |\cC_{e, k_i, i}|} & \le q_i\sum_{u \neq v \in S}|\cC_{\{u,v,w\},k_i,i}| \le k^2 \mu_{3, k_i, i}/\mu_{2, k_i, i} \\
& \le k^2/\bigpar{\Omega(n/k_i)p_i^2} \le n^{-1+\ca+o(1)} \ll n^{-2\ca}= \eps^2. 
\end{split}
\end{equation*}
We similarly obtain~$q_i \cdot |S| (1 + \eps) \mu_{2, k_i, i} = |S|/k^\cc \le 1$  
and~$q_i \le \bigpar{\Omega(n/k_i) p_i^{(k_i+1)/2}}^{-(k_i-2)} \le n^{-1+\ca+o(1)} \ll \deps$. 
Inserting~$1-q_i=e^{-(1+O(q_i))q_i}$ into~\eqref{eq:Ni:Exp:2}, 
using~$e^{o(\deps)}=1+o(\deps)$ it routinely follows~that
\begin{equation*}
		\P\lrpar{S \times \{w\} \subseteq E_{i+1}} 
		= \bigpar{1+o(\deps)} \cdot e^{-|S|/k^\cc}.
\end{equation*}
Recalling~\eqref{eq:Ni:Exp:1} and $(i+1) \deps \le I \deps \ll \eps \ll 1$, 
using~\eqref{eq:mainnibble} and~$\lambda_{|S|,i} \cdot e^{-|S|/k^\cc} = \lambda_{|S|,i+1}$ we infer that
\[
\E N_{S, i+1} = N_{S,i} \cdot \bigpar{1+o(\deps)} e^{-|S|/k^\cc} = \Bigpar{1 \pm \bigpar{i+1+o(1)} \deps } \cdot  \lambda_{|S|,i+1},
\]
which establishes~\eqref{eq:Ni:Exp} with room to spare, completing the proof of \refL{lem:Ni}. 
\end{proof}
For the interested reader we remark that the arguments of this section,  
and thus the proof of \refT{thm:mainpacking}, 
carry over to any (random or deterministic) initial graph~$G_0$ for which \refL{lem:N0} remains true.

\section{Random greedy edge coloring algorithm}\label{sec:edgechr}%
In this section we prove \refT{thm:ps:random} by showing that the following simple random greedy algorithm 
is likely to produce the desired proper edge coloring of the random edges from the hypergraph~$\cH$ (allowing for repeated edges), 
using the colors~$[q]=\{1, \ldots, q\}$ for suitable~$q \ge 1$. 
For $i\ge 0$, we sequentially choose an edge~$e_{i+1} \in E(\cH)$ uniformly at random, 
and then assign~$e_{i+1}$ a color~$c$ chosen uniformly at random from all colors in~$[q]$ that are still available at~$e_{i+1}$, 
i.e., which have not been assigned to an edge~$e_j$ with $e_j\cap e_{i+1}\neq \emptyset$ and $j\le i$ 
(this also ensures the usage of different colors for each occurrence of the same edge).  
This random greedy coloring algorithm terminates when no more colors are available at some edge~$e \in E(\cH)$.

\subsection{Dynamic concentration of key variables: proof of \refT{thm:ps:random}}\label{sec:dynamic}
Our main goal is to understand the evolution of the colors available for each edge~$e \in E(\cH)$, 
i.e., the size of~$Q_{e}(i)$, 
where for any set of vertices~$S\subseteq V(\cH)$ we more generally define 
\begin{equation}\label{def:QS}
Q_S(i) := \bigcpar{c \in [q] : \: \text{color $c$ not assigned to any 
edge~$f \in \{e_j: 1\le j\le i \}$ with~$f \cap S \neq \emptyset$}} .
\end{equation}
At the beginning of the algorithm we have~$|Q_{e}(0)|=q$. 
In order to keep track of the number of available colors~$|Q_{e}(i)|$, 
we need to understand changes in the colors assigned to edges adjacent to the vertices of~$e$. 
To take such changes into account, 
for all vertices $v \in V(\cH)$ and colors~$c \in [q]$ we introduce 
\begin{equation}\label{def:Yvc}
Y_{v,c}(i) := \bigcpar{f \in E(\cH): \: \text{$v \in f$ and $c \in Q_{f \setminus \{v\}}(i)$}},
\end{equation}
which in case of~$c \in Q_{\{v\}}(i)$ denotes the set of all edges adjacent to~$v$ that could still be colored by~$c$ 
(since for any $f \in Y_{v,c}(i)$ then~${c \in Q_{f \setminus \{v\}}(i) \cap Q_{\{v\}}(i) = Q_{f}(i)}$ holds).
Note that initially $|Y_{v,c}(0)| = \deg_{\cH}(v)$.

Our main technical result for the random greedy algorithm 
shows that, when $q \approx \kr m/n$ colors are used, 
then the above-mentioned key random variables closely follow the 
trajectories~$|Q_{e}(i)|\approx \hq(t)$ and $|Y_{v,c}(i)|\approx \hy(t)$
during the first ${m_0 \approx (1-\gamma)m}$ steps,  
tacitly using the continuous time scaling
\begin{equation}\label{def:t}
t = t(i,m) := i /m.
\end{equation}
In particular, $\min_{e \in E(\cH)}|Q_e(m_0)|>0$ 
ensures that the algorithm properly colors the first~$m_0$ edges using at most~$q$ colors, 
as no edge has run out of available colors.  
The form of the trajectories~\eqref{eq:Qi}--\eqref{def:hqye} 
 can easily be predicted via modern
(pseudo-random or expected one-step changes based) heuristics,  
see Appendix~\ref{sec:pseudo:int}. 
\begin{theorem}[Dynamic concentration of the variables]\label{thm:main:alg}
For all reals~$\gamma \in (0,1)$ and~$\sigma,\balp > 0$ with
\begin{equation}\label{cond:const}
\balp \log(1/\gamma) \le \sigma/30
\end{equation}
there is~$n_0=n_0(\sigma,\balp)>0$ such that, for all integers~${n \ge n_0}$, ${2 \le \kr \le \balp \log n}$ 
and all reals~${n^{1+\sigma} \le m \le n^{\kr n^{\sigma/4}}}$, ${D > 0}$, the following holds
for every $n$-vertex $\kr$-uniform hypergraph~$\cH$ satisfying 
the degree and codegree assumptions~\eqref{as:degcodeg}. 
With probability at least~${1 - m^{-\omega(\kr)}}$, we have ${\min_{e \in E(\cH)}|Q_e(i)|>0}$ and 
\begin{align}
\label{eq:Qi}
|Q_{e}(i)| \: &= \: \bigpar{1\pm \he(t)} \cdot \hq(t)     \hspace{-4.0em}&&\hspace{-4.0em}\text{for all $e \in E(\cH)$}, \\
\label{eq:Yuvi}
|Y_{v,c}(i)| \: &= \: \bigpar{1\pm \he(t)} \cdot \hy(t)    \hspace{-4.0em}&&\hspace{-4.0em}\text{for all $v \in V(\cH)$ and $c \in [q]$,}
\end{align}%
for all $0 \le i \le m_0:= \floor{\bigpar{1-\gamma} m}$, where~$q:=\floor{\kr m/n}$ and 
\begin{gather}
\label{def:hqye}
\hq(s) := (1-s)^\kr q,  
\qquad 
\hy(s) := (1-s)^{\kr-1} D 
\quad \text{ and } \quad 
\he(s) := (1-s)^{-9\kr} n^{-\sigma/3}.
\end{gather}
\end{theorem}
\begin{remark}\label{rem:ass}
The assumption~\eqref{cond:const} simply ensures that $\he(t) = (1-t)^{-9\kr}n^{-\sigma/3} \le n^{9\balp \log(1/\gamma)-\sigma/3} \le n^{-\sigma/30} = o(1)$ for all~$0 \le i \le m_0$, 
so that estimates~\eqref{eq:Qi}--\eqref{eq:Yuvi} imply~$|Q_{e}(i)| \sim \hq(t)$ and~$|Y_{v,c}(i)| \sim \hy(t)$. 
\end{remark}
\begin{remark}\label{rem:ass2}
The proof carries over to the case $\gamma=\gamma(n) \to 0$, 
provided that the assumption~\eqref{cond:const} is replaced by~$\kr \log(1/\gamma)/\log n \le \sigma/30$
(to again ensure that~$\he(t) \le n^{-\sigma/30} = o(1)$ holds). 
\end{remark}
Before giving the differential equation method based proof of this result, we first show how it implies Theorem~\ref{thm:ps:random} 
by slightly increasing the number of edges from~$m$ to~$m'$, 
to ensure that the greedy algorithm properly colors 
the first~$\floor{(1-\gamma)m'} \ge m$ random edges 
using at most~$\floor{\kr m'/n} \le (1+\eps)\kr m/n$ colors.  
\begin{proof}[Proof of Theorem~\ref{thm:ps:random}]
Set~$\gamma:=1-1/(1+\delta)$, so that~$\balp\log(1/\gamma)=\balp\log(1+1/\delta) \le \balp/\delta \le \sigma/30$ implies~\eqref{cond:const}. 
Invoking Theorem~\ref{thm:main:alg} with~$m$ set to~$m':=(1+\delta)m = o(n^{\kr n^{\sigma/4}})$ it follows that, with probability at least~${1 - m^{-\omega(\kr)}}$, 
the greedy algorithm properly colors the first $m_0 := {\floor{(1-\gamma)m'}=\floor{m}=m}$ random edges~$e_1, \ldots, e_m$  
using at most $q:={\floor{\kr m'/n} \le (1+\delta)\kr m/n}$ colors, 
completing the~proof.  
\end{proof}

\subsection{Differential equation method: proof of Theorem~\ref{thm:main:alg}}\label{sec:alg:proof}
In this subsection we prove Theorem~\ref{thm:main:alg} by showing ${\Pr(\neg\cG_{m_0}) \le m^{-\omega(\kr)}}$, 
where~$\cG_j$ denotes the event that ${\min_{e \in E(\cH)}|Q_e(i)|>0}$ 
and estimates~\eqref{eq:Qi}--\eqref{eq:Yuvi} hold for all ${0 \le i \le j}$.
We henceforth tacitly assume ${0 \le i \le m_0}$, 
and also that ${n \ge n_0(\sigma,\balp)}$ is sufficiently large (whenever necessary). 
In particular, estimate~\eqref{eq:Qi} implies ${\min_{e \in E(\cH)}|Q_e(i)|\ge \hq(t)/2 > 0}$ by \refR{rem:ass}. 
To establish~\eqref{eq:Qi}--\eqref{eq:Yuvi} 
using the differential equation method approach to dynamic concentration, 
we introduce the following sequences of auxiliary random variables:
\begin{align}
\label{eq:Qipm}
Q^{\pm}_e(i) & \: := \: \pm \bigsqpar{|Q_e(i)|-\hq(t)} - \he(t)\hq(t)    
\hspace{-2.0em}&&\hspace{-2.0em}\text{for all $e \in E(\cH)$}, \\
\label{eq:Yuvipm}
Y^{\pm}_{v,c}(i) & \: := \: \pm \bigsqpar{|Y_{v,c}(i)|- \hy(t)} - \he(t)\hy(t)    
\hspace{-2.0em}&&\hspace{-2.0em}\text{for all $v \in V(\cH)$ and $c \in [q]$.}
\end{align}
Note that the desired estimates~\eqref{eq:Qi}--\eqref{eq:Yuvi} follow 
when the four inequalities $Q^{\pm}_e(i) \le 0$ and $Y^{\pm}_{v,c}(i) \le 0$ all hold. 
To establish these inequalities, 
in Section~\ref{sec:alg:expected} we first estimate the expected one-step changes of~$|Q_e(i)|$ and~$|Y_{v,c}(i)|$, 
which in Section~\ref{sec:alg:super} then enables us to show that the sequences~$Q^{\pm}_e(i)$ and $Y^{\pm}_{v,c}(i)$ are supermartingales. 
Next, in Section~\ref{sec:alg:onestep} we bound the one-step changes of the variables, 
which in Section~\ref{sec:alg:estimates} then enables us to invoke a supermartingale inequality 
(that is optimized for the differential equation method, see Lemma~\ref{lem:superm}) 
in order to show that $Q^{\pm}_e(i) \ge 0$ or $Y^{\pm}_{v,c}(i) \ge 0$ are extremely unlikely~events.

\subsubsection{Expected one-step changes}\label{sec:alg:expected}
We first derive estimates for the expected one-step changes 
of the available colors variables~$|Q_e(i)|$ and the available edges variables~$|Y_{v,c}(i)|$, 
tacitly assuming that~${0 \le i \le m_0}$ and~$\cG_i$ hold. 
As we shall see, 
the expected changes~\eqref{eq:Qe:E} and~\eqref{eq:Yvc:E} will be consistent with the 
deterministic approximations~$|Q_e(i+1)|-|Q_e(i)| \approx \hq(t+1/m)-\hq(t) \approx \hq'(t)/m = -\kr (1-t)^{\kr-1} q/m$ 
and~$|Y_{v,c}(i+1)|-|Y_{v,c}(i)| \approx \hy'(t)/m=-(\kr-1) (1-t)^{\kr-2}D/m$, 
which is one motivation for the choice of~$\hq(t)$ and~$\hy(t)$; 
see also~\eqref{eq:Qe:heur:osc}--\eqref{eq:heur:fg} in~Appendix~\ref{sec:pseudo:int}. 

To calculate the expectation of the one-step changes~$\D Q_e(i) := |Q_e(i+1)|-|Q_e(i)|$, 
we consider a color~${c \in Q_e(i)}$ and the event that~$c \not\in Q_e(i+1)$. 
By definition~\eqref{def:QS} of $Q_e(i)$ this only occurs if the algorithm chooses an edge~$f$ with $f \cap e \neq \emptyset$, and then assigns the color~$c$ to~$f$. 
By definition~\eqref{def:Yvc} of~$Y_{v,c}(i)$ this color assignment is only possible if~${f \in \bigcup_{v \in e}Y_{v,c}(i)}$,
as~${c \in Q_{e}(i) \subseteq Q_{\{v\}}(i)}$ for any~${v \in e}$. 
Since the algorithm chooses both the edge ${e_{i+1} \in E(\cH)}$ and the color ${c \in Q_{e_{i+1}}(i)}$ uniformly at random, it follows~that
\begin{equation}\label{eq:Qe:E:0}
\E(\D Q_e(i) \mid \cF_{i}) = - \sum_{c \in Q_e(i)}\sum_{f \in \bigcup_{v \in e}Y_{v,c}(i)} \frac{1}{|E(\cH)| \cdot |Q_f(i)|} ,
\end{equation}
where $\cF_i$ denotes, as usual, the natural filtration associated with the algorithm after~$i$ steps (which intuitively keeps track of the history algorithm, i.e., contains all the information available up to step~$i$). 
Recalling the codegree assumption~\eqref{as:degcodeg} and~$\kr = O(\log n)$,
note that 
the cardinality of the union $\bigcup_{v \in e}Y_{v,c}(i)$ differs from the sum $\sum_{v \in e}|Y_{v,c}(i)|$ 
by at most $\sum_{v\neq w \in e}\deg_{\cH}(v,w) < n^{-\sigma/2}D < \he(t)\hy(t)$. 
The degree assumption~\eqref{as:degcodeg} also implies $\kr \cdot |E(\cH)|=\sum_{v \in V(H)}\deg_\cH(v) = n \cdot (1 \pm n^{-\sigma})D$. 
Using estimates~\eqref{eq:Qi}--\eqref{eq:Yuvi}, it follows that 
\begin{equation}\label{eq:Qe:E:1}
\E(\D Q_e(i) \mid \cF_{i}) = - \frac{(1\pm \he)\hq  \cdot \kr \cdot (1 \pm 2 \he)\hy }{(1 \pm n^{-\sigma})nD/\kr \cdot  (1 \pm \he)\hq } ,
\end{equation}
where we suppressed the dependence on~$t$ to avoid clutter in the notation. 
Noting ${|\kr m/n-q| \le 1 < n^{-\sigma} q}$ and ${n^{-\sigma} < \he(t) = o(1)}$, using $\hy(t) = (1-t)^{\kr-1} D$ we routinely arrive~at
\begin{equation}\label{eq:Qe:E}
\E(\D Q_e(i) \mid \cF_{i}) 
= - \bigpar{1\pm 7\he(t)} \cdot \kr (1-t)^{\kr-1} q/m . 
\end{equation}

To calculate the expectation of the one-step changes~$\D Y_{v,c}(i) := |Y_{v,c}(i+1)|-|Y_{v,c}(i)|$, 
we consider an edge~$f \in Y_{v,c}(i)$ and the event that~$f \not\in Y_{v,c}(i+1)$. 
By definition~\eqref{def:Yvc} of $Y_{v,c}(i)$ this only occurs if the algorithm chooses an edge~$e$ with $e \cap (f\setminus\{v\}) \neq \emptyset$, and then assigns the color~$c$ to~$e$, 
which in turn is only possible if~$e \in \bigcup_{w \in f \setminus\{v\}}Y_{w,c}(i)$. 
Proceeding similarly to~\eqref{eq:Qe:E:0}, it follows~that 
\begin{equation}\label{eq:Yvc:E:0}
\E(\D Y_{v,c}(i) \mid \cF_{i}) = - \sum_{f \in Y_{v,c}(i)}\sum_{e \in \bigcup_{w \in f \setminus\{v\}}Y_{w,c}(i)} \frac{1}{|E(\cH)| \cdot |Q_e(i)|} ,
\end{equation}
where~$|\bigcup_{w \in f \setminus\{v\}}Y_{w,c}(i)|$ differs from $\sum_{w \in f \setminus\{v\}}|Y_{w,c}(i)|$
by at most $\sum_{u\neq w \in f}\deg_{\cH}(u,w) < \he(t) \hy(t)$. 
Proceeding similarly to~\eqref{eq:Qe:E:1}--\eqref{eq:Qe:E}, 
using~$|q -\kr m/n| \le 1 < n^{-\sigma} \kr m/n$ and~${n^{-\sigma} < \he(t) = o(1)}$ 
it follows that 
\begin{equation}\label{eq:Yvc:E}
\E(\D Y_{v,c}(i) \mid \cF_{i}) = - \frac{(1\pm \he)\hy  \cdot (\kr-1) \cdot (1 \pm 2 \he)\hy }{(1 \pm n^{-\sigma})nD/\kr \cdot  (1 \pm \he)\hq } = - \bigpar{1\pm 7\he(t)} \cdot (\kr-1) (1-t)^{\kr-2}D/m . 
\end{equation}

\subsubsection{Supermartingale conditions}\label{sec:alg:super}
We now show that the expected one-step changes of the auxiliary variables $Q^{\pm}_e(i)$ and $Y^{\pm}_{v,c}(i)$ 
are negative (as~required for supermartingales), 
tacitly assuming that~${0 \le i \le m_0-1}$ and~$\cG_i$ hold. 
As we shall see, 
the main terms in the expected changes~\eqref{eq:Qepm:E:0} and~\eqref{eq:Yvcpm:E:0} 
will cancel due to the estimates of \refS{sec:alg:expected}, 
and the careful choice of~$\he(t)$ then ensures that the resulting expected changes~\eqref{eq:Qepm:E} and~\eqref{eq:Yvcpm:E} are indeed negative 
(by~ensuring that the ratios~$e'_X(t)/e_X(t)$ of the below-defined error functions~$e_X(t)$ are sufficiently~large).

For the one-step changes~$\D Q^{\pm}_e(i) := Q^{\pm}_e(i+1)-Q^{\pm}_e(i)$, set ${e_Q(s) := \he(s)\hq(s)} = {(1-s)^{-8\kr}n^{-\sigma/3}q}$. 
Recalling $t=i/m$, by applying Taylor's theorem with remainder 
it follows that 
\begin{equation}\label{eq:Qepm:E:0}
\begin{split}
\E(\D Q^{\pm}_e(i) \mid \cF_{i}) 
&= \pm \biggsqpar{\E(\D Q_e(i) \mid \cF_{i}) - \Bigsqpar{\hq\bigpar{t+1/m}-\hq(t)}} - \Bigsqpar{e_Q\bigpar{t+1/m} - e_Q(t)} \\
& = \pm \biggsqpar{\E(\D Q_e(i) \mid \cF_i) - \frac{\hq'(t)}{m}} - \frac{e'_Q(t)}{m}  + O\biggpar{\max_{s \in [0,m_0/m]}\frac{|\hq''(s)|+|e''_Q(s)|}{m^2}} .
\end{split}
\end{equation}
The key point is that the derivative~$\hq'(t)/m = -\kr (1-t)^{\kr-1} q/m$ equals the main term in~\eqref{eq:Qe:E}, and 
that the other term in~\eqref{eq:Qe:E} satisfies~$7\he(t) \cdot \kr (1-t)^{\kr-1} q/m = 7\kr (1-t)^{-1}e_Q(t) /m$. 
Furthermore, using the estimate from Remark~\ref{rem:ass} 
together with~$m \ge n^{1+\sigma}$ and $\kr = O(\log n)$, for all~$s \in [0,m_0/m]$ it is routine to see that 
\begin{equation}\label{eq:Qepm:error}
\begin{split}
\frac{|\hq''(s)|+|e''_Q(s)|}{m}\le O\biggpar{\frac{\kr^2 q + \kr^2 (1-s)^{-8\kr-2} n^{-\sigma/3}q}{m}} = o(n^{-\sigma/3}q) .
\end{split}
\end{equation}
Putting things together, now the crux is that~$e'_Q(t) = 8\kr(1-t)^{-1} e_Q(t) = \Omega(n^{-\sigma/3}q)$ implies 
\begin{equation}\label{eq:Qepm:E}
\begin{split}
\E(\D Q^{\pm}_e(i) \mid \cF_{i}) 
& \le \frac{7\kr (1-t)^{-1} e_Q(t) -e'_Q(t) + o(n^{-\sigma/3}q)}{m} < 0 .\end{split}
\end{equation}

For the one-step changes~$\D Y^{\pm}_{v,c}(i) := Y^{\pm}_{v,c}(i+1)-Y^{\pm}_{v,c}(i)$, set ${e_Y(s) := \he(s)\hy(s)} = {(1-s)^{-8\kr-1}n^{-\sigma/3}D}$. 
Proceeding similarly to~\eqref{eq:Qepm:E:0}, we obtain
\begin{equation}\label{eq:Yvcpm:E:0}
\begin{split}
\E(\D Y^{\pm}_{v,c}(i) \mid \cF_{i}) 
& = \pm \biggsqpar{\E(\D Y_{v,c}(i) \mid \cF_i) - \frac{\hy'(t)}{m}} - \frac{e'_Y(t)}{m}  + O\biggpar{\max_{s \in [0,m_0/m]}\frac{|\hy''(s)|+|e''_Y(s)|}{m^2}} .
\end{split}
\end{equation}
The key point is that the derivative~$\hy'(t)/m = -(\kr-1) (1-t)^{\kr-2} D/m$ equals the main term in~\eqref{eq:Yvc:E}, and 
that the other term in~\eqref{eq:Yvc:E} satisfies~$7\he(t) \cdot  (\kr-1) (1-t)^{\kr-2} D/m = 7(\kr-1) (1-t)^{-1}e_Y(t)/m$. 
Analogously to~\eqref{eq:Qepm:error}, it is routine to see that $|\hy''(s)|+|e''_Y(s)| = o(n^{-\sigma/3}Dm)$ for all~$s \in [0,m_0/m]$. 
Putting things together similarly to~\eqref{eq:Qepm:E},
here the crux is that~$e'_Y(t) = (8\kr+1)(1-t)^{-1} e_Y(t) = \Omega(n^{-\sigma/3}D)$ implies 
\begin{equation}\label{eq:Yvcpm:E}
\begin{split}
\E(\D Y^{\pm}_{v,c}(i) \mid \cF_{i}) 
& \le \frac{7(\kr-1) (1-t)^{-1} e_Y(t) -e'_Y(t) + o(n^{-\sigma/3}D)}{m} < 0 .\end{split}
\end{equation}

\subsubsection{Bounds on one-step changes}\label{sec:alg:onestep}
We next derive bounds on the one-step changes of the variables~$|Q_e(i)|$ and~$|Y_{v,c}(i)|$ 
(as~required by the supermartingale inequality 
in Section~\ref{sec:alg:estimates}), 
tacitly assuming that~${0 \le i \le m_0}$ and~$\cG_i$ hold. 
As we shall see, the expected changes~\eqref{eq:Qe:OSC:E} and~\eqref{eq:Yvc:OSC:E} 
are easy to bound due to step-wise monotonicity of the variables.

The one-step changes~$\D Q_e(i) = |Q_e(i+1)|-|Q_e(i)|$ of the available colors satisfy
\begin{equation}\label{eq:Qe:OSC:WC}
|\D Q_e(i)| \le 1. 
\end{equation}
Since~$|Q_e(i)|$ is step-wise decreasing, 
by inserting ${\he(t)=o(1)}$ and ${\kr (1-t)^{\kr-1} \le \kr}$ into~\eqref{eq:Qe:E} we obtain 
\begin{equation}\label{eq:Qe:OSC:E}
\E(|\D Q_e(i)| \mid \cF_{i}) = - \E(\D Q_e(i) \mid \cF_{i}) 
\le 2\kr q/m .
\end{equation}

The one-step changes~$\D Y_{v,c}(i) = |Y_{v,c}(i+1)|-|Y_{v,c}(i)|$ of the available edges satisfy 
\begin{equation}\label{eq:Yvc:OSC:WC}
|\D Y_{v,c}(i)| 
\le \sum_{w \in e_{i+1} \setminus\{v\}}\deg_\cH(v,w) \le \kr \cdot n^{-\sigma}D 
\end{equation}
due to the codegree assumption~\eqref{as:degcodeg}.
Since~$|\D Y_{v,c}(i)|$ is step-wise decreasing, 
by inserting ${\he(t)=o(1)}$ and ${(\kr-1) (1-t)^{\kr-2} \le \kr}$ into~\eqref{eq:Yvc:E} we also obtain
\begin{equation}\label{eq:Yvc:OSC:E}
\E(|\D Y_{v,c}(i)| \mid \cF_{i}) = - \E(\D Y_{v,c}(i) \mid \cF_{i}) \le 2\kr D/m .
\end{equation}

\subsubsection{Supermartingale estimates}\label{sec:alg:estimates}
We finally bound $\Pr(\neg\cG_{m_0})$ by focusing on the first step where the estimates~\eqref{eq:Qi}--\eqref{eq:Yuvi} are violated, 
which by the discussion below~\eqref{eq:Qipm}--\eqref{eq:Yuvipm} can only happen if $Q^{\pm}_e(i) \le 0$ or $Y^{\pm}_{v,c}(i) \le 0$ is violated. 
Our main tool for bounding the probabilities of these `first bad events' will be the following Freedman type supermartingale inequality: 
it is optimized for the differential equation method approach to dynamic concentration,
where supermartingales~$S_i$ are constructed by adding a deterministic quantity to a random variable~$X_i$, 
cf.~the definition of~$Q^{\pm}_e(i)$ and $Y^{\pm}_{v,c}(i)$ in~\eqref{eq:Qipm}--\eqref{eq:Yuvipm}. 
Here the convenient point is that \refL{lem:superm} only assumes upper bounds on the one-step changes of~$X_i$ 
(and not of~$S_i$, as usual, cf.~\cite[Lemma~3.4]{bohman2019large}).   
\begin{lemma}\label{lem:superm}
Let~$(S_i)_{i \ge 0}$ be a supermartingale adapted to the filtration $(\cF_i)_{i\ge 0}$. 
Assume that~${S_i=X_i+D_i}$, where $X_i$ is~$\cF_i$-measurable and $D_i$ is~$\cF_{\max\{i-1,0\}}$-measurable.
Writing ${\D X_i:=X_{i+1}-X_i}$, assume that ${\max_{i \ge 0}|\D X_i| \le C}$ and ${\sum_{i \ge 0} \E(|\D X_i| \mid \cF_{i}) \le V}$.  
Then, for all~$z >0$, 
\begin{equation}
\label{eq:superm}
\Pr\bigpar{S_i \ge S_0 + z \text{ for some $i \ge 0$}}
\: \le \:
\exp \biggpar{-\frac{z^2}{2C(V+z)}}. 
\end{equation}
\end{lemma}
\begin{proof}
Writing~$\D S_i:=S_{i+1}-S_i$, set~$M_i:=S_i-\sum_{0\le j<i}\E (\Delta S_j\mid \cF_j)$. 
Note that $S_i=X_i+D_i$ implies
\[ \D M_i:=M_{i+1}-M_i=\D S_i-\E(\D S_i\mid \cF_i)=\D X_i -\E(\D X_i\mid \cF_i),   \]
which readily gives $\E(\Delta M_i\mid \mathcal{F}_i)=0$ and $\max_{i\ge 0}|\Delta M_i|\le 2 \cdot C$.
Note that we also have
\begin{equation}\label{eq:var:Mi}
\Var(\D M_i\mid \cF_i)
=\Var(\D X_i\mid \cF_i)
\le \E(\D X_i^2\mid \cF_i)
\le C\cdot \E(|\D X_i|\mid \cF_i) ,  
\end{equation}
so that $\sum_{i\ge 0}\Var(\D M_i\mid \cF_i)\le C \cdot V$. 
Clearly~${M_0=S_0}$. 
Also~$M_i\ge S_i$, since $(S_i)_{i\ge 0}$ is a supermartingale. 
Hence a standard application of Freedman's martingale inequality (see~\cite{freedman1975tail} or~\cite[Lemma~2.2]{warnke2016method}) 
yields 
\begin{equation}
\label{eq:superm2}
\Pr\bigpar{S_i\ge S_0+z\text{ for some }i\ge 0}
\le \Pr\bigpar{M_i\ge M_0+z \text{ for some $i\ge 0$}}
\le \exp\biggpar{-\frac{z^2}{2(CV+2C \cdot z/3)}},   
\end{equation}
which completes the proof of inequality~\eqref{eq:superm}.
\end{proof}

Turning to the details, we define the stopping time~$I$ as the minimum of~$m_0$ and the first step~$i \ge 0$ where~$\cG_i$ fails. 
Writing~$i \wedge I := \min\set{i,I}$, as usual, by our above discussion it follows that 
\begin{equation}\label{eq:union:Gm}
\begin{split}
\Pr\bigpar{\neg \cG_{m_0}} &\le \sum_{e \in E(\cH)}\sum_{\tau \in \set{+,-}}\Pr\bigpar{Q^{\tau}_{e}(i \wedge I) \ge 0 \text{ for some $i \ge 0$}}\\
& \quad 
+ \sum_{v \in V(\cH)}\sum_{c \in [q]}\sum_{\tau \in \set{+,-}}\Pr\bigpar{Y^{\tau}_{v,c}(i \wedge I) \ge 0 \text{ for some $i \ge 0$}}. 
\end{split}
\end{equation}
Note that initially~$|Q_e(i)|=q$ and~$|Y_{v,c}(0)|=\deg_\cH(v)$, 
which in view of the degree assumption~\eqref{as:degcodeg} 
and the definitions~\eqref{eq:Qipm}--\eqref{eq:Yuvipm} 
of~$Q^{\tau}_e(0)$ and $Y^{\tau}_{v,c}(0)$ 
gives the initial value estimates 
\begin{gather*}
Q^{\tau}_{e}(0 \wedge I) = Q^{\tau}_e(0) 
= - \he(0) q = -n^{-\sigma/3}q,\\
Y^{\tau}_{v,c}(0 \wedge I) = 
Y^{\tau}_{v,c}(0) 
= O(n^{-\sigma}D) - \he(0) D \le 
-n^{-\sigma/3}D/2 .
\end{gather*}
Noting that the estimates from Sections~\ref{sec:alg:super}--\ref{sec:alg:onestep} apply for~${0 \le i \le I-1}$ (since then ${0 \le i \le m_0-1}$ and $\cG_i$ hold), 
the point is that the stopped sequence~${S_i := Q^{\tau}_{e}(i \wedge I)}$ is a supermartingale with $S_0 = -n^{-\sigma/3}q$, to which Lemma~\ref{lem:superm} can be applied with ${X_i = \tau |Q_e(i \wedge I)}|$, $C=1$ and $V=m_0 \cdot 2\kr q/m  = O(\kr q)$. 
Invoking inequality~\eqref{eq:superm} with~$z=n^{-\sigma/3}q$, 
using~$q = \Omega(\kr n^{\sigma})$ together with $m^{\kr}\le n^{\kr^2 n^{\sigma/4}}$ and~$\kr = O(\log n)$ it follows~that 
\begin{equation}\label{eq:superm:Qe}
\Pr\bigpar{Q^{\tau}_{e}(i \wedge I) \ge 0 \text{ for some $i \ge 0$}} \le \exp\Bigcpar{-\Theta(n^{-2\sigma/3}q/\kr)}
\le \exp\Bigcpar{-\Theta(n^{\sigma/3})} 
\le m^{-\omega(\kr)}.
\end{equation}
Similarly, the sequence~${S_i := Y^{\tau}_{v,c}(i \wedge I)}$ is a supermartingale with $S_0 \le -n^{-\sigma/3}D/2$, to which Lemma~\ref{lem:superm} can be applied with ${X_i = \tau |Y_{v,c}(i \wedge I)}|$, $C=\kr n^{-\sigma}D $ and $V= m_0 \cdot 2\kr D/m  = O(\kr D)$.
Invoking inequality~\eqref{eq:superm} with~$z=n^{-\sigma/3}D/2$, 
it follows analogously to~\eqref{eq:superm:Qe} that 
\begin{equation}\label{eq:superm:Yvc}
\Pr\bigpar{Y^{\tau}_{v,c}(i \wedge I) \ge 0 \text{ for some $i \ge 0$}} \le \exp\Bigcpar{-\Theta(n^{\sigma/3}/\kr^2)} \le m^{-\omega(\kr)} .
\end{equation}
Inserting~\eqref{eq:superm:Qe}--\eqref{eq:superm:Yvc} into inequality~\eqref{eq:union:Gm}, 
noting~$|V(H)| = n \le m$, $|E(\cH)| \le n^\kr \le m^\kr$ and $q \le m$ 
it then follows that~$\Pr(\neg\cG_{m_0}) \le m^{-\omega(\kr)}$, 
which completes the proof of Theorem~\ref{thm:main:alg}.%
\noproof

%

\section{Concluding remarks}\label{sec:concl} 
%
The main remaining open problem is to determine the typical asymptotic behavior 
of the Prague dimension~${\fpra(\Gnp) \approx \ccc(\Gnpp{1-p})}$ as well as the 
clique covering and partition numbers~$\ccn(\Gnp)$ and~$\cpn(\Gnp)$, 
i.e., to refine the estimates from Theorems~\ref{thm:prague}, \ref{thm:packing} and~\ref{thm:thickCI}. 
Here edge-probability~$p=1/2$ is of special interest, since this would reveal the asymptotics 
of these intriguing parameters for almost all $n$-vertex~graphs. 
\begin{problem}\label{prm:open}
Determine the whp asymptotics of the 
parameters~$\ccn(\Gnp)$, $\cpn(\Gnp)$, $\cct(\Gnp)$, and~$\ccc(\Gnp)$ 
for constant edge-probabilities~$p \in (0,1)$. 
\end{problem}

\subsection{Non-trivial lower bounds for dense random graphs}\label{subsec:lowerdense}
For constant edge-probabilities~${p \in (0,1)}$ our understanding of the asymptotics remains unsatisfactory, even on a heuristic level. 
Indeed, it is well-known that the largest clique of~$\Gnp$ whp has size ${s \sim 2\log_{1/p} n}$, 
which together with the simple lower bound reasoning for
\refT{thm:packing} makes it tempting to speculate that perhaps 
${\ccn(\Gnp) \sim \binom{n}{2}p/\binom{s}{2}}$ holds~whp. 
However, \refL{lemma:lowerbounds} shows that this natural guess is false, 
by further improving the simple lower bound (which for~$p=1/2$ was already noted in~\cite{BESW1993}). 
The analogous speculation ${\cct(\Gnp) \sim np/(s-1)}$ is also refuted by \refL{lemma:lowerbounds}, 
whose proof we defer to~\refApp{apx:lower}. 
Perhaps rashly, we speculate that the lower bounds in~\eqref{eq:ccn:lb}--\eqref{eq:cct:lb} 
could perhaps be asymptotically best possible. 
\begin{lemma}\label{lemma:lowerbounds}
If~$p=p(n)$ satisfies~$n^{-o(1)} \le p \le 1-n^{-o(1)}$, 
then for any $\eps \in (0,1)$ whp  
\begin{align}
\label{eq:ccn:lb}
\ccn(\Gnp) \; &\ge \; (1-\eps) \cdot \bigpar{ 1+\varphi(p)} \tbinom{n}{2}p/\tbinom{s}{2} ,\\
\label{eq:cct:lb}
\cct(\Gnp) \; &\ge \; (1-\eps) \cdot \bigpar{ 1+\varphi(p)} np/(s-1) ,
\end{align}
where~$s:=\ceil{2\log_{1/p} n}$ and~$\varphi(p):={(1-p)\log(1-p)/(p\log p)}$.  
The function $\varphi:(0,1) \to (0,\infty)$ is increasing, with 
${\lim_{p \searrow 0}\varphi(p)=0}$, ${\varphi(1/2)=1}$, and ${\lim_{p  \nearrow 1}\varphi(p)=\infty}$. 
\end{lemma} 

\subsection{Asymptotics for sparse random graphs}\label{subsec:asymptoticsmallp}
We now record strengthenings of Theorems~\ref{thm:packing}--\ref{thm:thickCI} 
for many small edge-probabilities~${p=p(n) \to 0}$,  
where the asymptotics follow from Pippenger--Spencer type hypergraph results.  
As we shall see, here the crux is that when all cliques have size~$O(1)$, 
then it suffices to simply cover a~$1-o(1)$ fraction of the relevant~edges. 
%
%
\begin{theorem}\label{thm:packing:const}
If~$p=p(n)$ satisfies~$n^{-2/(s+1)} \ll p \ll n^{-2/(s+2)}$ for some fixed integer~$s \ge 3$, 
then $\ccn(\Gnp)$ and $\cpn(\Gnp)$ are whp both asymptotic to~$\binom{n}{2}p/\binom{s}{2}$. 
\end{theorem}
%
We leave it as an open problem to determine the whp asymptotics for~$p=\Theta(n^{-2/(s+1)})$, 
and now outline the proof of \refT{thm:packing:const}, which uses~$\ccn(\Gnp) \le \cpn(\Gnp)$. 
The lower bound on $\ccn(\Gnp)$ is routine: 
the expected number of edges in cliques of size at least~$s+1$ is at most 
${\sum_{k \ge s+1}\binom{k}{2}\binom{n}{k}p^{\binom{k}{2}}} \ll {\binom{n}{2}p}$, 
which makes it easy to see that whp~$\ccn(\Gnp) \ge (1-o(1)) \binom{n}{2}p/\binom{s}{2}$. 
For the upper bound on~$\cpn(\Gnp)$ we shall mimic the natural strategy of Kahn and Park~\cite{kahn2020tuza} for~$s=3$: 
using Kahn's fractional version of Pippenger's hypergraph packing result~\cite[Theorem~7.1]{kahn2020tuza}  
it is not difficult\footnote{We consider the auxiliary hypergraph~$\cH$,  
where the vertices correspond to the edges of~$\Gnp$ and the edges correspond to the edge-sets of the cliques~$K_s$ of~$\Gnp$. 
The technical conditions of~\cite[Theorem~7.1]{kahn2020tuza} 
required for mimicking~\cite[Section~7]{kahn2020tuza} can then be verified 
using (careful applications of) standard tail bounds such as~\refL{lem:UT} and~\cite[Theorems~30~and~32]{warnke2020missing}.} 
to see that~$\Gnp$ whp contains a collection~$\cC$ of~$|\cC|=(1-o(1))\binom{n}{2}p/\binom{s}{2}$ edge-disjoint cliques~$K_s$. 
Writing~$\cU$ for the edges of~$\Gnp$ not covered by the cliques in~$\cC$, 
it then easily follows that whp~$\cpn(\Gnp)  \le |\cC| + |\cU| \le (1+o(1))\binom{n}{2}p/\binom{s}{2}$, 
as~desired.

\begin{theorem}\label{thm:index:const}
If~$p=p(n)$ satisfies~$(\log n)^{\omega(1)}n^{-2/(s+1)}\le p \ll n^{-2/(s+2)}$ for some fixed integer~$s \ge 3$, 
then $\cct(\Gnp)$ and $\ccc(\Gnp)$ are whp both asymptotic to~$np/(s-1)$. 
\end{theorem}
\begin{remark}
These asymptotics remain valid when the definitions of $\cct(\Gnp)$ and $\ccc(\Gnp)$ are restricted to clique partitions of the edges (instead of clique coverings). 
\end{remark}
%
We leave it as an open problem to determine the whp asymptotics for~$p=(\log n)^{O(1)}n^{-2/(s+1)} $,  
and now outline the proof of \refT{thm:index:const}, which uses~$\cct(\Gnp)\le \ccc(\Gnp)$. 
The lower bound on $\cct(\Gnp)$ is routine: 
the expected number of edges in cliques of size at least~$s+1$ containing a fixed vertex~$v$ is at most 
${\sum_{k \ge s+1}\binom{k}{2}\binom{n-1}{k-1}p^{\binom{k}{2}}} \ll {np}$, 
which makes it easy to see that whp~$\cct(\Gnp) \ge (1-o(1)) np/(s-1)$. 
Turning to the upper bound on~$\ccc(\Gnp)$, 
using a pseudo-random variant of Pippenger's packing result due to Ehard, Glock and Joos~\cite{ehard2019pseudorandom}, 
it is not difficult\footnote{For the same auxiliary hypergraph~$\cH$ as considered before, 
the required technical conditions of~\cite[Theorem~1.2]{ehard2019pseudorandom}  with $\Delta \approx \binom{n-2}{s-2}p^{\binom{s}{2}-1} \ge \Omega\bigpar{(\log n)^{\omega(1)}}$ and~$\log e(\cH) \le s \log n \ll \Delta^{\Theta(1)}$ can be verified using \refL{lem:UT} and~\cite[Theorem~1]{vsileikis2019counting}.}
to see that~$\Gnp$ whp contains a collection~$\cC$ of edge-disjoint cliques~$K_s$ 
where each vertex is contained in~$(1-o(1))np/(s-1)$ cliques of~$\cC$. 
Writing~$\cU$ for the edges of~$\Gnp$ not covered by the cliques in~$\cC$, 
using Pippenger and Spencer's chromatic index result~\cite{PS1989} and Vizing's theorem 
it then is not difficult to see that whp~$\ccc(\Gnp) \le \chi'(\cC) + \chi'(\cU) \le (1+o(1)) n p/(s-1)$, 
as~desired.

\small
\bibliographystyle{plain}

\normalsize

\appendix

\section{Lower bounds: proof of \refL{lemma:lowerbounds}}\label{apx:lower}
%
%
\begin{proof}[Proof of \refL{lemma:lowerbounds}]%
Writing~$\cS$ for the event that the largest clique of~$\Gnp$ has size at most~$s=\ceil{2\log_{1/p} n}$, it well-known that~$\cS$~holds~whp 
(by a straightforward first moment argument). 
Writing~$\cE$ for the event that~$\Gnp$ contains $(1\pm\eps)\binom{n}{2}p$ edges  
for~$\eps :=n^{-1/2}$, say, 
it is easy to see that~$\cE$~holds~whp (using Chebychev's inequality). 
Furthermore, 
recalling $\varphi(p)={(1-p)\log(1-p)/(p\log p)}$, 
the probability that~$\Gnp$ equals any fixed spanning 
subgraph~$G \subseteq K_n$ with~$e(G) = (1\pm\eps)\binom{n}{2}p$ edges
 is routinely seen to be at~most 
\begin{equation}\label{eq:Pi}
\Pi := 
\max_{m \in (1\pm\eps)\binom{n}{2}p} p^m (1-p)^{\binom{n}{2}-m} 
\; \le \; \exp\Bigpar{-(1-o(1)) \! \cdot \! \tbinom{n}{2}p \bigpar{1+\varphi(p)} \! \cdot \! \log(1/p)} .
\end{equation}

For the clique covering number~$\ccn(\Gnp)$, the crux is that there are at~most 
\[
\binom{n+s}{s}^T 
\: \le \: 
o(n^{sT})
\]
many collections~$\{C_1, \ldots, C_t\}$ with $t \le T$ that are a 
clique covering for some graph~$G \subseteq K_n$ with largest clique of size at most~$s$. 
Hence, since each clique covering uniquely determines the entire edge-set 
and thus the underlying spanning subgraph~$G \subseteq K_n$, 
it follows by a union bound argument~that 
\begin{equation}\label{eq:theta0:lower:dense:pr}
\Pr(\ccn(\Gnp) \le T) \; \le \; \P(\neg\cS \text{ or } \neg \cE) \: + \: o(n^{sT}) \cdot \Pi. 
\end{equation}
Note that~$\P(\neg\cS \text{ or } \neg \cE)=o(1)$ and~$s\log n \sim \binom{s}{2} \cdot \log(1/p)$. 
In view of inequality~\eqref{eq:Pi}, for any~$\eps \in (0,1)$ 
it follows that~\eqref{eq:theta0:lower:dense:pr} is at most~$o(1)$
when~$T \le (1-\eps) \cdot (1+\varphi(p)) \binom{n}{2}p/\binom{s}{2}$, 
establishing~\eqref{eq:ccn:lb}. 

Turning to the thickness~$\cct(\Gnp)$, we associate each clique covering~$\cC$ of some graph~$G \subseteq K_n$ 
with an auxiliary bipartite graph~$\cB$ on vertex-set ${[n] \cup \cC}$,  
where~${v \in [n]}$ and~${C_i \in \cC}$ are connected by an edge whenever~${v \in V(C_i)}$.
If the thickness of~$\cC$ is at most~$T$, then in~$\cB$ the degree of each~$v \in [n]$ is at most~$\floor{T}$, 
which also gives~${|\cC| \leq n \floor{T}}$.   
Since the structure of the auxiliary bipartite graph~$\cB$ uniquely determines~$\cC$  
(as the neighbors of~$C_i$ in~$\cB$ determine the clique vertex-set~$V(C_i)$), 
it follows that there are at~most
\[
\binom{n\floor{T} + \floor{T}}{\floor{T}}^n
\: \le \: 
O\bigpar{(6n)^{nT}}
\]
many collections~$\cC$ with thickness at most~$T$ that are a clique covering of some graph~$G \subseteq K_n$. 
Since each such~$\cC$ 
uniquely determines the underlying spanning subgraph~$G \subseteq K_n$, 
we obtain similarly to~\eqref{eq:theta0:lower:dense:pr}~that
\begin{equation}\label{eq:theta1:lower:dense:pr}
\Pr(\cct(\Gnp) \le T) \; \le \; \P(\neg \cE) \: + \: O\bigpar{(6n)^{nT}} \cdot \Pi. 
\end{equation}
Note that~$\P(\neg \cE)=o(1)$ and~$n \log (6n) \sim  \binom{n}{2}\log(1/p) \cdot (s-1)/n$. 
In view of inequality~\eqref{eq:Pi}, for any~$\eps \in (0,1)$ 
it follows that~\eqref{eq:theta1:lower:dense:pr} is at most~$o(1)$ 
when~$T \le (1-\eps) \cdot (1+\varphi(p)) np/(s-1)$, completing the proof of~\eqref{eq:cct:lb}. 
\end{proof}

\section{Variant of \refT{thm:ps:random}: proof of \refC{c:ps:random}}\label{apx:lower}
\begin{proof}[Proof of \refC{c:ps:random}]
Choosing~$\xi=\xi(\delta) \in (0,1/16]$ such that $(1+\delta)(1+\xi)/(1-4\xi)^2 \le 1+2\delta$,  
set ${m_0 := \floor{(1+\xi)m}}$, ${m_1 := \floor{m_0/(1-4\xi)^2}}$, and ${c := (1+\delta)r m_1/n}$. 
Let $\cH_{i}^*$ be chosen uniformly at random from all $\binom{|E(\cH)|}{i}$ subhypergraphs of~$\cH$ with exactly~$i$ edges. 
Since $\cH_q$ conditioned on having exactly~$i$ edges has the same distribution as~$\cH_i^*$, 
by the law of total probability and monotonicity it follows~that 
\begin{equation}\label{eq:cl:ps:random:transfer}
\begin{split}
\Pr(\ci(\cH_q) \ge c) & \: \le \: \Pr(|E(\cH_q)| > m_0) + \sum_{0 \le i \le m_0}\Pr(\ci(\cH_i^*) \ge c)\Pr(|E(\cH_q)| = i) \\
& \: \le \: n^{-\omega(r)} + \Pr(\ci(\cH_{m_0}^*) \ge c), 
\end{split}
\end{equation}
where we used standard Chernoff bounds (such as~\cite[Theorem~2.1]{JLR}) and~$\E |E(\cH_q)| = |E(\cH)|q = m \ge n^{1+\sigma} \gg r \log n$.  
Sequentially choosing the random edges $e_1, \ldots, e_{m_1} \in E(\cH)$ of~$\cH_{m_1}$ as defined in \refT{thm:ps:random}, 
note that $e_{i+1} \in {E(\cH) \setminus \{e_1, \ldots, e_i\}}$ holds with probability at least~${1-m_1/e(\cH) > 1-4\xi}$, 
as~$m_1 < 4 m \le 4\xi e(\cH)$.  
Since we can equivalently construct the edge-set~$\{f_1, \ldots, f_{m_0}\}$ of $\cH^*_{m_0}$ by 
sequentially choosing $f_{i+1} \in {E(\cH) \setminus \{f_1, \ldots, f_i\}}$ uniformly at random, 
a natural coupling of~$\cH_{m_1}$ and~$\cH^*_{m_0}$ thus satisfies 
\begin{equation*}
\begin{split}
\Pr(\cH_{m_0}^* \subseteq \cH_{m_1}) \: \ge \: \Pr(\Bin(m_1,1-4\xi) \ge m_0) \: \ge \: 1- n^{-\omega(r)} ,
\end{split}
\end{equation*}
where we used standard Chernoff bounds and that~$m_1 (1-4\xi) > m_0/(1-\xi)$ for~$n \ge n_0(\xi)$.  
Hence 
\begin{align}\label{eq:cl:ps:random:2}
\Pr(\ci(\cH_{m_0}^*) \ge c) \: \le \: \Pr(\ci(\cH_{m_1}) \ge c) + n^{-\omega(r)} \: \le \: n^{-\omega(r)} ,
\end{align}
where we invoked \refT{thm:ps:random} with~$m$ set to $m_1$ 
(which applies since ${n^{1+\sigma} \le m} \le m_1 < {4 \xi e(\cH) < n^r}$).  
This completes the proof by combining~\eqref{eq:cl:ps:random:transfer} and~\eqref{eq:cl:ps:random:2} with~$c \le (1+2\delta)rm/n$. 
\end{proof}

\section{Heuristics: random greedy edge coloring algorithm}\label{sec:pseudo:int}
In this appendix we give, for the greedy coloring algorithm from \refS{sec:edgechr}, 
two heuristic explanations for the trajectories~$|Q_{e}(i)| \approx \hq(t)$ and $|Y_{v,c}(i)| \approx \hy(t)$ that these random variables follow, where ${t = t(i,m) = i /m}$. 

For our first \emph{pseudo-random heuristic}, 
we write~$E_i=\{e_1, \ldots, e_i\}$ for the multi-set of edges appearing during the first $i$ steps of the algorithm. 
Ignoring that edges can appear multiple times, 
our pseudo-random ansatz is that 
the edges in~$E_i$ and their assigned colors 
are approximately independent~with 
\begin{equation*}
\Pr\bigpar{\text{$e$ in $E_i$ and colored~$c$}} \approx \frac{|E_i|}{|E(\cH)|} \cdot \frac{1}{q} \approx \frac{i}{nD/\kr} \cdot \frac{1}{\kr m/n}
= \frac{t}{D} =: p(t,D) = p ,
\end{equation*}
where independence only holds with respect to colorings that are proper, i.e., possible in the algorithm. 
Using this heuristic ansatz, we now consider the event~$\cE_{v,c}$ 
that no edge~$f \in E_i$ with~$v \in f$ is colored~$c$. 
Exploiting that no two distinct edges containing~$v$ can receive the same color in the algorithm (since this coloring would not be proper), 
our pseudo-random ansatz and the degree assumption~\eqref{as:degcodeg} then suggests that 
\begin{equation*}
\Pr(\neg\cE_{v,c}) = \sum_{f \in E(\cH): v \in f} \hspace{-0.25em} \Pr(\text{$f$ in $E_i$ and colored~$c$}) \approx D \cdot p = t .
\end{equation*}
Since for every pair of vertices there are only at most~$n^{-\sigma}D$ edges containing both (by the codegree assumption), 
for~$\ell= o(\log n)$ distinct vertices~$v_{1},\ldots,v_{\ell}$ 
our pseudo-random ansatz also loosely suggests that 
\begin{equation*}
\Pr\Bigpar{ \bigcap_{i \in [\ell]}\cE_{v_i,c}} 
 \approx \prod_{i \in [\ell]}\Pr(\cE_{v_i,c}) + O\bigpar{\ell^2 \cdot n^{-\sigma}D \cdot p}\approx (1-t)^\ell .
\end{equation*}
Recalling~\eqref{def:QS} from \refS{sec:edgechr}, using linearity of expectation we then anticipate~$|Q_{e}(i)| \approx \hq(t)$ based on  
\begin{equation*}
\E |Q_{e}(i)| 
= \sum_{c \in [q]}\Pr\bigpar{c \in Q_{e}(i)} 
= \sum_{c \in [q]}\Pr\Bigpar{ \bigcap_{v \in e}\cE_{v,c}} 
\approx q \cdot (1-t)^\kr = \hq(t) .
\end{equation*}
Mimicking this reasoning, 
recalling~\eqref{def:Yvc} 
we similarly anticipate~$|Y_{v,c}(i)| \approx \hy(t)$ based on
\begin{equation*}
\E |Y_{v,c}(i)| 
= \sum_{f \in E(\cH): v \in f} \hspace{-0.25em} \Pr\bigpar{c \in Q_{f \setminus \{v\}}(i)} 
\approx D \cdot (1-t)^{\kr-1} = \hy(t). 
\end{equation*}

In our second \emph{expected one-step changes heuristic} 
we assume for simplicity that there are deterministic approximations
$|Q_{e}(i)| \approx f(t) q$ and $|Y_{v,c}(i)| \approx g(t) D$. 
Using these approximations and~$q \approx \kr m/n$, 
the calculations leading to \eqref{eq:Qe:E:0}--\eqref{eq:Qe:E:1} and \eqref{eq:Yvc:E:0}--\eqref{eq:Yvc:E} 
in Section~\ref{sec:alg:expected} 
then suggest~that 
\begin{align}
\label{eq:Qe:heur:osc}
\E\bigpar{|Q_{e}(i+1)|-|Q_{e}(i)| \: \big| \: \cF_i} & 
\approx - \frac{f(t) q \cdot \kr \cdot g(t)D}{nD/\kr \cdot f(t) q} 
\approx -\frac{\kr g(t)q}{m} , \\
\label{eq:Yvc:heur:osc}
\E\bigpar{|Y_{v,c}(i+1)|-|Y_{v,c}(i)| \: \big| \: \cF_i} & 
\approx - \frac{g(t)D \cdot (\kr-1) \cdot g(t)D}{nD/\kr \cdot f(t) q} 
\approx -\frac{(\kr-1)g^2(t)D}{f(t) m} ,
\end{align}
where $\cF_i$ denotes the natural filtration of the algorithm after $i$ steps.  
Since the left-hand sides of~\eqref{eq:Qe:heur:osc}--\eqref{eq:Yvc:heur:osc} are approximately equal to~$[f(t+1/m)-f(t)]q \approx  f'(t)q/m$ and~$g'(t)D/m$, respectively, 
we anticipate 
\begin{equation}\label{eq:heur:fg}
f'(t) = -\kr g(t) 
\quad \text{ and } \quad
g'(t) = -(\kr-1) g^2(t)/f(t) .
\end{equation}
Noting~$|Q_{e}(0)| = q$ and $|Y_{v,c}(0)| \approx D$, we also anticipate~$f(0)=g(0)=1$.
The solutions ${f(t)=(1-t)^\kr}$ and ${g(t)=(1-t)^{\kr-1}}$ then make 
$|Q_{e}(i)| \approx f(t) q = \hq(t)$ 
and 
$|Y_{v,c}(i)| \approx g(t) D = \hy(t)$
plausible.

\end{document}